\documentclass[11pt]{article}

\usepackage{a4wide}
\usepackage{amssymb,amsmath,amsthm, amsfonts}

\usepackage{color}

\newenvironment{keywords}{{\bf Key words. }}{\\}
\newenvironment{AMS}{{\bf AMS subject classification. }}{\\}

\newcommand{\at}[1]{}
\newcommand{\email}[1]{{\tt #1}}

\newcommand{\lin}{{\rm lin\,}}

\newcommand{\Gr}{{\rm gph\,}}

\newcommand{\co}{{\rm conv\,}}
\newcommand{\ri}{{\rm ri\,}}
\newcommand{\cl}{{\rm cl\,}}

\newcommand{\xb}{\bar x}
\newcommand{\yb}{\bar y}
\newcommand{\zb}{\bar z}

\newcommand{\lb}{\bar\lambda}
\newcommand{\yba}{\yb^\ast}
\newcommand{\Tb}{\bar{\cal T}}
\newcommand{\Cb}{\bar{\cal C}}
\newcommand{\Q}{{\cal Q}}
\newcommand{\Pa}{{\cal P}}
\newcommand{\I}{{\cal I}}
\newcommand{\Ib}{\bar\I}

\newcommand{\Kb}{\bar K}
\newcommand{\Lb}{\bar \Lambda}

\newcommand{\R}{\mathbb{R}}

\newcommand{\dist}[1]{{\rm d}(#1)}

\newcommand{\mv}{\,\vert\, }
\newcommand{\setto}[1]{\mathop{\to}\limits^#1}

\newcommand{\skalp}[1]{\langle #1\rangle}
\newcommand{\range}{{\rm Range\,}}

\newcommand{\Span}{\mathop{\rm span\,}\limits}

\newcommand{\Tlin}{T^{\rm lin}_\Omega}
\newcommand{\TD}{T_D(F(\xb))}
\newcommand{\NrD}{\widehat N_D(F(\xb))}

\newtheorem{definition}{Definition}
\newtheorem{theorem}{Theorem}
\newtheorem{lemma}{Lemma}
\newtheorem{example}{Example}
\newtheorem{corollary}{Corollary}
\newtheorem{proposition}{Proposition}
\newtheorem{remark}{Remark}
\newtheorem{assumption}{Assumption}

\begin{document}

\title{On estimating the regular normal cone to constraint systems and stationarity conditions\thanks{This is an Accepted Manuscript of an article published by Taylor \& Francis in Optimization on 31 October 2016, available online: http://www.tandfonline.com/10.1080/02331934.2016.1252915}}
\author{Mat\'u\v{s} Benko, Helmut Gfrerer\thanks{Institute of Computational Mathematics, Johannes Kepler University Linz,
              A-4040 Linz, Austria, \email{benko@numa.uni-linz.ac.at},
             \email{helmut.gfrerer@jku.at}}
}

\date{}

\maketitle
\begin{abstract}Estimating the regular normal cone to constraint systems plays an  important role for the derivation of sharp necessary optimality conditions. We present two novel approaches and introduce a new stationarity concept which is stronger than M-stationarity. We apply our theory to three classes of mathematical programs frequently arising  in the literature.
\end{abstract}

\section{Introduction}
This paper deals with the computation of the {\em regular normal cone}  $\widehat N_\Omega(\xb)$ to sets of the form
\begin{equation}\label{EqConstrSyst}
    \Omega:=\{x\in \R^n\mv F(x)\in D\}
\end{equation}
at some point $\xb\in\Omega$, where $F:\R^n\to\R^m$ is a mapping continuously differentiable  at $\xb$ and $D\subset\R^m$ is a closed set.

This task is of particular importance for the development of first order optimality conditions of the nonlinear program
\begin{equation}\label{EqOptProbl}
  \min f(x)\quad\text{ subject to }\quad x\in\Omega
\end{equation}
since the basic optimality condition, see e.g. \cite[Theorem 6.12]{RoWe98}, states that the negative gradient of the objective at a local minimizer $\xb$ belongs to the regular normal cone to the constraints at $\xb$, i.e.
\[-\nabla f(\xb)\in\widehat N_\Omega(\xb).\]

When $D$ is convex, the computation of the regular normal cone is well understood, see e.g. \cite{BonSh00}. Under some constraint qualification condition an exact formula  reads as
\begin{equation}\label{EqNormalConeConv}
\widehat N_\Omega(\xb)=\nabla F(\xb)^T\widehat N_D(F(\xb)).
\end{equation}
Quite more complicated is the situation, when $D$ is not convex. This occurs for instance, when among the constraints so-called equilibrium constraints are present. Such programs are usually termed {\em mathematical programs with equilibrium constraints} (MPEC). The equilibrium can be often described by a lower-level optimization problem, by variational inequalities or by complementarity constraints. Some of these equilibrium constraints can be written as smooth equalities and inequalities, but these constraints usually do not satisfy the common constraint qualifications of nonlinear programming. Alternative formulations yield either a nonsmooth mapping or the system \eqref{EqConstrSyst} with nonconvex $D$, the case considered in this paper. Prominent examples are {\em mathematical programs with complementarity constraints} (MPCC) or {\em mathematical programs with vanishing constraints} (MPVC). We refer the reader to the paper \cite{Sch04} for some more examples on this subject.

In  case when $D$ is not convex, only inclusions for the regular normal cone are known in general. The lower estimate is given by
\begin{equation}\label{EqLowerEstim}
\nabla F(\xb)^T\widehat N_D(F(\xb))\subset \widehat N_\Omega(\xb)
\end{equation}
and is known to hold with equality, if the Jacobian  $\nabla F(\xb)$ has full rank, cf. \cite[Example 6.7]{RoWe98}.  When we have equality in \eqref{EqLowerEstim}, the corresponding optimality conditions are usually called {\em S-stationarity} (strong stationarity) conditions in the literature on mathematical programs with equilibrium constraints (MPECs). The main drawback of the S-stationarity conditions is the requirement of strong constraint qualification conditions.

If one weakens the used constraint qualification condition then the inclusion \eqref{EqLowerEstim} will be strict in general. In this situation one has to consider an upper estimate to the regular normal cone  $\widehat N_\Omega(\xb)$. A commonly used upper estimate is provided by the so-called {\em limiting normal cone} to $\Omega$ at $\xb$.  The use of the limiting normal cone has the advantage, that a lot of calculus rules are available for its calculation; we refer the readers to the textbooks \cite{Mo06a,Mo06b,RoWe98}. Optimality conditions based on this upper estimate involving the limiting normal cone are usually called {\em M-stationarity} conditions. A main disadvantage of this approach is, that in general the regular normal cone is strictly included in the limiting normal cone. Therefore, in general M-stationarity does not preclude the existence of feasible descent directions.

The aim of this paper is to provide  estimates to the regular normal cone $\widehat N_\Omega(\xb)$ which are valid under very weak constraint qualification conditions and are tighter than the one based on the limiting normal cone.

For this purpose  we present two new approaches. The first one is motivated by a result due to Pang and Fukushima \cite{PaFu99} and yields an upper bound for the regular normal cone which is exact under some suitable assumptions. This upper estimate for the regular normal cone  constitutes a new stationarity concept called $\Q_M$-stationarity which is shown to be stronger than M-stationarity. We apply this approach to MPCC and improve the result due to Pang and Fukushima \cite{PaFu99}. For MPVC we derive a new qualification condition, which resembles the well known {\em Mangasarian Fromovitz constraint qualification} (MFCQ) of nonlinear programming, and allows the exact computation of the regular normal cone for MPVC. The obtained results are much stronger than the known results from literature \cite{AchKa08,DoShSt12,HoKa08,HoKa09,HoKaOut10,IzSo09}. Finally we analyze MPECs where the constraints are given by a generalized equation (GE) involving the normal cone mapping to $C^2$ inequalities together with parameter constraints. Again we derive upper bounds for the regular normal cone which can be exact under certain conditions and can be employed to replace the commonly used  conditions as in \cite[Theorem 3.4]{HenOutSur12}.

In the second approach treated in this paper we focus on the lower inclusion \eqref{EqLowerEstim} for the regular normal cone and state a condition which ensures equality. This new condition is  an extension of the recent result \cite[Theorem 4]{GfrOut14a} and we apply it also to MPECs with an additional parameter constraint.

The paper is organized as follows. In section 2 we present some basic definitions and results from variational analysis together with the definitions of various stationarity concepts. In section 3 we give the theoretical background for the two approaches presented in this paper for estimating the regular normal cone as well as the new concepts of $\Q$-stationarity and $\Q_M$-stationarity, respectively. In sections 4, 5 and 6 we apply the results from section 3 to MPCC, MPVC and an MPEC, respectively.

Our notation is basically standard. $K^\circ$ stands for the polar to a cone $K$ and $\Span\{u_1,\ldots,u_N\}$ stands for the subspace generated by the vectors $u_1,\ldots,u_N$. By $\nabla F(\xb)$ we normally denote the Jacobian of the mapping $F$ at $\xb$, but occasionally we use it like a linear mapping to write
\[\nabla F(\xb)^{-1}Q:=\{u\mv \nabla F(\xb)u\in Q\}\]
for a set $Q$. To ease the notation the Minkowski sum of a singleton $\{a\}$ and a set $A$ is denoted by $a+A$.

\section{Preliminaries}
Let us start with geometric objects. Given a set
$\Gamma\subset\R^d$ and a point $\bar z\in\Gamma$, define
the (Bouligand-Severi) {\em tangent/contingent cone} to $\Gamma$
at $\bar z$  by
\begin{equation}\label{tan}
T_\Gamma(\bar z):=\Big\{u\in\R^d\mv\exists\,t_k\searrow
0,\;u_k\to u\;\mbox{ with }\;\bar z+t_k u_k\in\Gamma ~\forall ~ k\Big\}.
\end{equation}
Note that one has $T_\Gamma(\bar z)=\R_+(\Gamma-\bar z)$ when $\Gamma$ is a convex polyhedron.

The (Fr\'{e}chet) {\em regular normal cone} to $\Gamma$ at $\bar
z\in\Gamma$ can be defined as the polar cone to the tangent cone by
\begin{equation}\label{fn}
\widehat N_\Gamma(\bar
z):=(T_\Gamma(\bar z))^\circ.
\end{equation}
Further, the (Mordukhovich) {\em limiting/basic normal cone} to $\Gamma$ at $\bar
z\in\Gamma$ is given by
\begin{equation}\label{EqLimNormalCone}N_\Gamma(\bar z):=\{z^\ast\mv \exists\ z_k\setto\Gamma \bar z,\ z_k^\ast\to z^\ast\mbox{ with } z_k^\ast\in \widehat N_\Gamma(z_k)~\forall~k\}.
\end{equation}
Note that the
tangent/contingent cone and the regular normal cone reduce to
the classical tangent cone and normal cone of convex analysis,
respectively, when the set $\Gamma$ is convex.
We put $T_\Gamma(\zb)=\widehat N_\Gamma(\zb)=N_\Gamma(\zb)=\emptyset$, if $\zb\not\in\Gamma$. Note that we always have
\[\widehat N_\Gamma(\bar z)\subset N_\Gamma(\bar z).\]

Next we recall some rules for calculating polar cones. For two closed convex cones $C_1$ and $C_2$ we have
\[(C_1\cup C_2)^\circ=(C_1+C_2)^\circ=C_1^\circ\cap C_2^\circ,\quad (C_1\cap C_2)^\circ =\cl(C_1^\circ+C_2^\circ)\]
and for closed convex cones $P_j,Q_j$, $j=1,\ldots,m$ we have
\begin{equation}
  \label{EqPolarConeProd}
  \Big(\prod_{i=1}^mP_i\Big)^\circ\cap \Big(\prod_{i=1}^mQ_i\Big)^\circ=\Big(\prod_{i=1}^mP_i^\circ\Big)\cap \Big(\prod_{i=1}^mQ_i^\circ\Big)=
  \prod_{i=1}^m\left(P_i^\circ\cap Q_i^\circ\right)=\prod_{i=1}^m\left(P_i\cup Q_i\right)^\circ.
\end{equation}

\begin{proposition}\label{PropPolarCones}
 Let $A$ be an $s\times d$ matrix, let $C\subset \R^s$ be a cone and assume that either there exists some $u$ such that $Au\in\ri\co C$ or $C$ is polyhedral, i.e. $C$ is the union of finitely many convex polyhedral cones $C_1,\ldots,C_p$. Then
        \begin{equation}\label{EqPolarA}
          \{u\mv Au\in \co C\}^\circ = A^T C^\circ
        \end{equation}
\end{proposition}
\begin{proof}
     In case when there exists some $u$ with $Au\in\ri\co C$, the statement follows from \cite[Corollary 16.3.2]{Ro70}. Now consider the case when $C$ is polyhedral. Then $\co C=\sum_{i=1}^p C_i$ is a convex polyhedral set by \cite[Corollary 19.3.2]{Ro70} and so is its polar $(\co C)^\circ=C^\circ=\bigcap_{i=1}^p C_i^\circ$ by \cite[Corollary 19.2.2]{Ro70}. By virtue of \cite[Theorem 19.3]{Ro70} the set $A^TC^\circ$ is again convex and polyhedral and now the statement follows from \cite[Corollary 16.3.2]{Ro70} by taking into account that convex polyhedral sets are always closed.
\end{proof}
\begin{lemma}\label{LemIntersect}
  Let $A$ be an $s\times d$ matrix and let $S_1,S_2\subset \R^d$ be two sets. Then
  \[(AS_1)\cap (AS_2)= A(S_1\cap(\ker A+S_2)).\]
\end{lemma}
\begin{proof}
  If $z\in (AS_1)\cap (AS_2)$, then there are $s_1\in S_1$, $s_2\in S_2$ with $z=As_1= As_2$. Since $s_1=s_2+(s_1-s_2)$ and $A(s_1-s_2)=0$, the properties $s_1\in S_1\cap(\ker A+S_2)$ and $z\in A(S_1\cap(\ker A+S_2))$ follow. Conversely, if $z\in A(S_1\cap(\ker A+S_2))$, then there are $s_1\in S_1$, $s_2\in S_2$ and $r\in \ker A$ such that $s_1=r+s_2$ and $z=As_1\in AS_1$. It follows that $z=A(r+s_2)=As_2\in AS_2$ and thus $z\in (AS_1)\cap (AS_2)$.
\end{proof}
We now introduce generalizations of the  Abadie constraint qualification condition and the Guignard constraint qualification condition, respectively, as known from nonlinear programming.

\begin{definition}
Let $\Omega$ be given by \eqref{EqConstrSyst} and let $\xb\in\Omega$.
\begin{enumerate}
\item We say that the {\em generalized Abadie constraint qualification} (GACQ) holds at $\xb$ if
\begin{equation}\label{EqGACQ}
  T_\Omega(\xb)=\Tlin(\xb),
\end{equation}
where $\Tlin(\xb):=\{u\in\R^n\mv  \nabla F(\xb)u\in T_D(F(\xb))\}$ denotes the {\em linearized  cone}.
\item We say that the {\em generalized Guignard constraint qualification} (GGCQ) holds at $\xb$ if
\begin{equation}\label{EqGGCQ}
  (T_\Omega(\xb))^\circ=(\Tlin(\xb))^\circ.
\end{equation}
\end{enumerate}
\end{definition}

Obviously GGCQ is weaker than GACQ, but GACQ is easier to verify because several advanced methods from variational analysis are available. To this end we need the concepts of metric regularity and metric subregularity of multifunctions.

\begin{definition}Let $\Psi:\R^d\rightrightarrows\R^s$ be a multifunction,  $(\bar u,\bar v)\in\Gr \Psi$ and
 $\kappa>0$. Then
\begin{enumerate}
\item $\Psi$ is called {\em metrically regular with modulus
$\kappa$} near $(\bar u,\bar v)$ if there are neighborhoods $U$ of
$\bar u$ and $V$ of $\bar v$ such that
\begin{equation}
\label{EqMetrReg} \dist{u,\Psi^{-1}(v)}\leq
\kappa\dist{v,\Psi(u)}\ \forall (u,v)\in U\times V.
\end{equation}
\item $\Psi$ is called {\em metrically subregular with modulus
$\kappa$} at $(\bar u,\bar v)$ if there is a neighborhood $U$ of
$\bar u$ such that
\begin{equation}
\label{EqMetrSubReg} \dist{u,\Psi^{-1}(\bar v)}\leq
\kappa\dist{\bar v,\Psi(u)}\ \forall u\in U.
\end{equation}
\end{enumerate}
\end{definition}

It is well known that metric regularity of the multifunction
$\Psi$ near $(\bar u,\bar v)$ is equivalent to the Aubin property (also called Lipschitz-like or pseudo-Lipschitz)
of the inverse multifunction $\Psi^{-1}$ and metric subregularity
of $\Psi$ at $(\bar u,\bar v)$ is equivalent with the property of
{\em calmness} of its inverse.

Obviously, metric regularity of $\Psi$ near $(\bar u,\bar v)$ implies metric subregularity of $\Psi$ at $(\bar u,\bar v)$.

\begin{proposition}[{cf.\cite[Proposition 1]{HenOut05}}]\label{PropGACQ}
Let $\xb$ belong to the set $\Omega$ given by \eqref{EqConstrSyst}.
 If the perturbation mapping
\begin{equation}
  \label{EqPertMapping}M(x):=F(x)-D
\end{equation}
associated with the constraint system \eqref{EqConstrSyst} is metrically subregular at $(\xb,0)$, then $GACQ$ holds at $\xb$.
\end{proposition}

Metric regularity of the mapping \eqref{EqPertMapping} can be verified by the so-called {\em Mordukhovich criterion}, see, e.g., \cite[Example 9.44]{RoWe98}.
Tools for verifying metric subregularity of constraint systems can be found e.g. in \cite{GfrKl16}.

The following theorem states some fundamental relations between the regular and the limiting normal cone.
\begin{theorem}
  \label{ThInclNormalCone}
  Let $\Omega$ be given by \eqref{EqConstrSyst} and let $\xb\ \in \Omega$. Then
  \begin{equation}
    \label{EqInclRegNormalCone} \nabla F(\xb)^T\widehat N_D(F(\xb))\subset\widehat N_\Omega(\xb).
  \end{equation}
  On the other hand, if the multifunction \eqref{EqPertMapping} is metrically subregular at $(\xb,0)$ then
  \begin{equation}
    \label{EqInclLimNormalCone} N_\Omega(\xb)\subset \nabla F(\xb)^TN_D(F(\xb)).
  \end{equation}
  If $\nabla F(\xb)$ has full rank, then both inclusions \eqref{EqInclRegNormalCone} and \eqref{EqInclLimNormalCone} hold with equality.
\end{theorem}
\begin{proof}
  The inclusion \eqref{EqInclRegNormalCone} can be found in \cite[Theorem 6.14]{RoWe98}, whereas \eqref{EqInclLimNormalCone} follows from \cite[Theorem 4.1]{HenJouOut02}. For the statement on equality in the inclusions we refer to \cite[Exercise 6.7]{RoWe98}.
\end{proof}

At the end of this section we consider different stationarity concepts.
\begin{definition}
  Let $\xb$ be feasible for the program \eqref{EqOptProbl}, where $\Omega$ is given by \eqref{EqConstrSyst} and $f$ is assumed to be smooth.
  \begin{enumerate}
  \item We say that $\xb$ is {\em B-stationary (Bouligand stationary)} if
  \[0\in \nabla f(\xb)+\widehat N_\Omega(\xb).\]
  \item We say that $\xb$ is {\em S-stationary (strongly stationary)} if
  \[0\in \nabla f(\xb)+\nabla F(\xb)^T\widehat N_D(F(\xb)).\]
  \item We say that $\xb$ is {\em M-stationary (Mordukhovich stationary)} if
  \[0\in \nabla f(\xb)+\nabla F(\xb)^T N_D(F(\xb)).\]
  \end{enumerate}
\end{definition}
By the definition of the regular normal cone we have
\[\skalp{\nabla f(\xb),u}\geq 0\ \forall u\in T_\Omega(\xb)\]
at a B-Stationary point, which expresses that no feasible descent direction exists. Every local minimizer is known to be B-stationary. Conversely, if $\xb$ is B-stationary then there exists some smooth mapping $\hat f:\R^n\to\R$ with $\nabla\hat f(\xb)=\nabla f(\xb)$ such that $\xb$ is a global minimizer of the problem $\min _{x\in\Omega}\hat f(x)$, cf. \cite[Theorem 6.11]{RoWe98}.

From \eqref{EqInclRegNormalCone} it is easy to see that every S-stationary point is also B-stationary, but the reverse statement is not true in general, unless we have equality in \eqref{EqInclRegNormalCone}.

On the other hand, a B-stationary point $\xb$ is also M-stationary provided that the perturbation mapping $M$ is metrically subregular at $(\xb,0)$. However, M-stationarity does not preclude the existence of feasible descent directions, unless we have $\widehat N_\Omega(\xb)=N_\Omega(\xb)=\nabla F(\xb)^T N_D(F(\xb))$.

Since we have $\widehat N_\Omega(\xb)\subset  N_\Omega(\xb)$ by the definition, we derive  from Theorem \ref{ThInclNormalCone} the inclusion
\[\widehat N_\Omega(\xb)\subset \nabla F(\xb)^T N_D(F(\xb)).\]
under the assumption of metric subregularity of \eqref{EqPertMapping} at $(\xb,0)$. This relation can be strengthened by the following proposition.
\begin{proposition}\label{PropImprInclLimNormalCone}Let $\Omega$ be given by \eqref{EqConstrSyst}, let $\xb\ \in \Omega$ and assume that GGCQ is fulfilled, while the mapping $u\rightrightarrows\nabla  F(\xb)u-\TD$ is metrically subregular at $(0,0)$. Then
\[\widehat N_\Omega(\xb)\subset  \nabla F(\xb)^T N_{\TD}(0)\subset \nabla F(\xb)^T N_D(F(\xb)).\]
\end{proposition}
\begin{proof}
  By virtue of GGCQ we have $\widehat N_{\Omega}(\xb)= (\Tlin(\xb))^\circ=\widehat N_{\Tlin(\xb)}(0)$ and since $u\rightrightarrows\nabla  F(\xb)u-\TD$ is assumed to be metrically subregular at $(0,0)$, we can apply Theorem \ref{ThInclNormalCone} to obtain
  $\widehat N_{\Tlin(\xb)}(0)\subset N_{\Tlin(\xb)}(0)\subset \nabla F(\xb)^T N_{\TD}(0)$. By \cite[Proposition 6.27]{RoWe98} we have $N_{\TD}(0)\subset N_D(F(\xb))$ and this finishes the proof.
\end{proof}
If $\TD$ is the union of finitely many convex polyhedral cones, then the mapping $u\rightrightarrows\nabla  F(\xb)u-\TD$ is a polyhedral multifunction and consequently metrically subregular at $(0,0)$ by Robinson's result \cite{Rob81}. Hence we arrive at the following corollary which slightly improves \cite[Theorem 7]{FleKanOut07}.
\begin{corollary}\label{CorLinMstat}Let $\xb$ be B-stationary for the program \eqref{EqOptProbl}, where $\Omega$ is given by \eqref{EqConstrSyst} and $f$ is assumed to be smooth. If GGCQ is fulfilled at $\xb$ and $\TD$ is the union of finitely many convex polyhedral cones, then $\xb$ is M-stationary and even the stronger condition
\[0\in\nabla f(\xb)+ \nabla F(\xb)^T N_{\TD}(0)\]
holds.
\end{corollary}

\section{\label{SecTheory}Estimating the regular normal cone}

Throughout this section we assume that the set $\Omega$ is given by $\eqref{EqConstrSyst}$, where $F:\R^n\to\R^m$ is continuously differentiable at the reference point $\xb\in\Omega$ and $D\subset\R^m$ is closed. Further we assume that the objective $f:\R^n\to\R$ of the program \eqref{EqOptProbl} is continuously differentiable at  $\xb$ and GGCQ holds.

The main goal of this section is to provide a tight estimate for the regular normal cone $\widehat N_\Omega(\xb)$, which, thanks to GGCQ, amounts to $(\Tlin(\xb))^\circ$. To this end we discuss two possibilities, the first one being motivated by the paper of Pang and Fukushima \cite{PaFu99}  is based on the following observation.
\begin{theorem}\label{ThQ} Let $Q_1$ and $Q_2$ denote two closed convex cones contained in $\TD$. If
\begin{equation}\label{EqAssQ}
  (\nabla F(\xb)^{-1}Q_i)^\circ=\nabla F(\xb)^TQ_i^\circ,\ i=1,2
\end{equation}
then
\begin{equation}
  \label{EqInclRegNormalConeQ1_Q2}
   \widehat N_\Omega(\xb)\subset \nabla F(\xb)^T\left(Q_1^\circ\cap(\ker\nabla F(\xb)^T+ Q_2^\circ)\right)=(\nabla F(\xb)^TQ_1^\circ)\cap(\nabla F(\xb)^TQ_2^\circ).
\end{equation}
Further, if
\begin{equation}\label{EqEqualityBetaStat}\nabla F(\xb)^T\left(Q_1^\circ\cap(\ker\nabla F(\xb)^T+ Q_2^\circ)\right)\subset \nabla F(\xb)^T\NrD,
\end{equation}
 then equality holds in \eqref{EqInclRegNormalConeQ1_Q2}.
\end{theorem}
\begin{proof}
Since $\nabla F(\xb)^{-1}Q_i\subset \nabla F(\xb)^{-1}\TD=\Tlin(\xb)$, $i=1,2$ we have
\begin{eqnarray*}\widehat N_\Omega(\xb)&=&(\Tlin(\xb))^\circ\subset(F(\xb)^{-1}Q_1\cup\nabla F(\xb)^{-1} Q_2)^\circ =(F(\xb)^{-1}Q_1)^\circ\cap(\nabla F(\xb)^{-1}Q_2)^\circ\\
&=&\nabla F(\xb)^TQ_1^\circ\cap \nabla F(\xb)^TQ_2^\circ\end{eqnarray*}
and \eqref{EqInclRegNormalConeQ1_Q2} follows from Lemma \ref{LemIntersect}.
 To show the sufficiency  of condition \eqref{EqEqualityBetaStat} for equality in \eqref{EqInclRegNormalConeQ1_Q2}, note that condition \eqref{EqEqualityBetaStat} together with \eqref{EqInclRegNormalConeQ1_Q2} implies $\widehat N_\Omega(\xb)\subset \nabla F(\xb)^T\NrD$. Now, equality in \eqref{EqInclRegNormalConeQ1_Q2} follows from \eqref{EqInclRegNormalCone}.
\end{proof}
The proper choice of $Q_1$ and $Q_2$ is crucial in order that \eqref{EqInclRegNormalConeQ1_Q2} provides a good estimate for the regular normal cone. It is obvious that we want to choose the cones $Q_i$, $i=1,2$ as large as possible in order that the inclusion \eqref{EqInclRegNormalConeQ1_Q2} is tight. Further it is reasonable that a good choice of $Q_1,Q_2$ fulfills
\begin{equation}
  \label{EqNecEqual}Q_1^\circ\cap Q_2^\circ=\NrD
\end{equation}
because then condition \eqref{EqEqualityBetaStat} holds whenever $\nabla F(\xb)$ has full rank.

Since $Q_i\subset \TD$, we  have $Q_i^\circ\supset(\TD)^\circ=\NrD$, $i=1,2$ and consequently, $Q_1^\circ\cap(\ker\nabla F(\xb)^T+ Q_2^\circ)\supset Q_1^\circ\cap Q_2^\circ\supset \NrD$.  Hence the inclusion \eqref{EqEqualityBetaStat} can never be strict.

The following definition is motivated by Theorem \ref{ThQ}.
\begin{definition}\label{DefQStat}Let $\cal Q$ denote some collection of pairs $(Q_1,Q_2)$ of closed convex cones fulfilling
\begin{equation}\label{EqCondQi}Q_i\subset \TD, (\nabla F(\xb)^{-1}Q_i)^\circ=\nabla F(\xb)^TQ_i^\circ, i=1,2.
\end{equation}

{\bf(i)} Given $(Q_1,Q_2)\in {\cal Q}$ we say that $\xb$ is {\em ${\cal Q}$-stationary with respect to $(Q_1,Q_2)$} for the program \eqref{EqOptProbl}, if
\[0\in\nabla f(\xb)+ \nabla F(\xb)^T\left(Q_1^\circ\cap(\ker\nabla F(\xb)^T+ Q_2^\circ)\right).\]

{\bf(ii)} We say that $\xb$ is {\em ${\cal Q}$-stationary} for the program \eqref{EqOptProbl}, if  $\xb$ is $\Q$-stationary with respect to some pair $(Q_1,Q_2)\in{\cal Q}$.

{\bf(iii)} We say that $\xb$ is $\Q_M$-stationary, if there exists a pair $(Q_1,Q_2)\in {\cal Q}$ such that
\[0\in\nabla f(\xb)+\nabla F(\xb)^T\left(Q_1^\circ\cap(\ker\nabla F(\xb)^T+ Q_2^\circ)\cap N_D(F(\xb))\right).\]
\end{definition}

The following corollary follows immediately from the definitions and Theorem \ref{ThQ}.
\begin{corollary}\label{CorBetaStat}
  Assume that $\xb$ is B-stationary for the program \eqref{EqOptProbl}. Then $\xb$ is $\Q$-stationary with respect to every pair $(Q_1,Q_2)\in{\cal Q}$. Conversely, if $\xb$ is $\Q$-stationary with respect to some pair $(Q_1,Q_2)\in{\cal Q}$ fulfilling condition \eqref{EqEqualityBetaStat}, then $\xb$ is S-stationary and consequently, also B-stationary.
\end{corollary}

The following lemma follows immediately from \eqref{EqInclRegNormalConeQ1_Q2} and the definition of $\Q$-stationarity.
\begin{lemma}\label{LemBetaStat}Let $(Q_1,Q_2)\in {\cal Q}$. Then  $\xb$ is $\Q$-stationary with respect to $(Q_1,Q_2)$ for the program \eqref{EqOptProbl} if and only if $-\nabla f(\xb)\in \nabla F(\xb)^TQ_i^\circ$, $i=1,2$.
\end{lemma}
\begin{corollary}
  \label{CorSStatImplBStat}Let $\xb$ be S-stationary for the program \eqref{EqOptProbl}. Then  $\xb$ is $\Q$-stationary with respect to every $(Q_1,Q_2)\in {\cal Q}$.
\end{corollary}
\begin{proof}Since $Q_i\subset \TD$, we have $\widehat N_D(F(\xb))\subset Q_i^\circ$, $i=1,2$. Hence S-stationarity of $\xb$ implies
\[-\nabla f(\xb)\in \nabla F(\xb)^T\widehat N_D(F(\xb))\subset \nabla F(\xb)^TQ_i^\circ\]
and the assertion follows from Lemma \ref {LemBetaStat}.

\end{proof}
\begin{remark}Note that for $i=1,2$ the program
\[(P_i)\qquad \min_{u\in\R^n}\nabla f(\xb)u\quad \text{subject to}\quad \nabla F(\xb)u\in Q_i\]
is a convex program  and therefore the first-order optimality condition
\[-\nabla f(\xb)\in N_{\nabla F(\xb)^{-1}Q_i}(0)=(\nabla F(\xb)^{-1}Q_i)^\circ=\nabla F(\xb)^TQ_i\]
is both necessary and sufficient in order that $u=0$ is a solution of $(P_i)$. Hence $\xb$ is $\Q$-stationary with respect to $(Q_1,Q_2)$ if and only if $0$ is a solution for the programs $(P_1)$ and $(P_2)$, respectively.
\end{remark}

By the definition, a $\Q_M$-stationary point is both M-stationary and $\Q$-stationary. However, a B-stationary point is $\Q_M$-stationary only under some additional condition. This is due to the fact that under the assumptions of Theorem \ref{ThInclNormalCone} we have
\[\widehat N_\Omega(\xb)\subset \nabla F(\xb)^T\left(Q_1^\circ\cap(\ker\nabla F(\xb)^T+ Q_2^\circ)\right)\cap \nabla F(\xb)^TN_D(F(\xb))\ \forall (Q_1,Q_2)\in{\cal Q},\] but in general
\begin{eqnarray*}\lefteqn{\nabla F(\xb)^T\left(Q_1^\circ\cap(\ker\nabla F(\xb)^T+ Q_2^\circ)\right)\cap \nabla F(\xb)^TN_D(F(\xb))}\qquad\\
\qquad&\not=&\nabla F(\xb)^T\left(Q_1^\circ\cap(\ker\nabla F(\xb)^T+ Q_2^\circ)\cap N_D(F(\xb))\right).
\end{eqnarray*}
Clearly, equality holds when $\nabla F(\xb)$ possesses full row rank, but in this case a B-stationary point is already S-stationary. In the following theorem we state three more sufficient conditions ensuring $\Q_M$ stationarity of a B-stationary point.
\begin{theorem}\label{ThM_betaStat}
  Assume that $\xb$ is B-stationary for the program \eqref{EqOptProbl}. Then  $\xb$ is $\Q_M$-stationary if  any of the following three conditions holds:
  \begin{enumerate}
    \item There exists a pair $(Q_1,Q_2)\in{\cal Q}$ such that
    \begin{equation}\label{EqBeta_M_Suff1} Q_1^\circ\cap(\ker\nabla F(\xb)^T+ Q_2^\circ)\subset N_D(F(\xb)).\end{equation}
   \item $\xb$ is M-stationary and for every $\lambda\in N_D(F(\xb))$ there is some pair $(Q_1,Q_2)\in{\cal Q}$ with $\lambda\in Q_1^\circ.$
   \item $\TD$ is the union of finitely many convex polyhedral sets and  for every $t\in \TD$ there is some pair $(Q_1,Q_2)\in \Q$ satisfying $t\in Q_1$.
  \end{enumerate}
\end{theorem}
\begin{proof}
  Under the condition \eqref{EqBeta_M_Suff1}, $\Q_M$-stationarity of $\xb$ follows immediately from the definition and Corollary \ref{CorBetaStat}.
  Let us prove the second case. Since $\xb$ is  M-stationary, there exists some $\lambda\in N_D(F(\xb))$ verifying $-\nabla f(\xb)=\nabla F(\xb)^T\lambda$ and by the assumption there is some $(Q_1,Q_2)\in{\cal Q}$ with $\lambda\in Q_1^\circ$ implying $-\nabla f(\xb)\in \nabla F(\xb)^T\big(Q_1^\circ\cap N_D(F(\xb))\big)$. By using that  $\xb$ is B-stationary and therefore also $\Q$-stationary with respect to $(Q_1,Q_2)$ by Corollary \ref{CorBetaStat}, by virtue of Lemmas \ref{LemBetaStat} and \ref{LemIntersect} we obtain
  \begin{eqnarray}\nonumber-\nabla f(\xb)&\in& \nabla F(\xb)^T\big(Q_1^\circ\cap N_D(F(\xb))\big)\cap \nabla F(\xb)^TQ_2^\circ \\
  \label{EqAuxQM_Stat}&=&   \nabla F(\xb)^T\big(Q_1^\circ\cap N_D(F(\xb))\cap (\ker \nabla F(\xb)^T+Q_2^\circ)\big)
  \end{eqnarray}
  showing $\Q_M$-stationarity of $\xb$. Now let us prove  the sufficiency of the third condition. By Corollary \ref{CorLinMstat} there is some $\lambda\in N_{\TD}(0)$ with $-\nabla f(\xb)=\nabla F(\xb)^T\lambda$ and by using  \cite[Lemma 3.4]{Gfr13b}, we can find some $t\in \TD$ with $\lambda\in \widehat N_{\TD}(t)$. Our assumption guarantees that there is some pair $(Q_1,Q_2)\in\Q$ with $t\in Q_1\subset\TD$ and therefore
  $\lambda\in\widehat N_{\TD}(t)\subset \widehat N_{Q_1}(t)=Q_1^\circ\cap \{t\}^\perp\subset Q_1^\circ$ by convexity of $Q_1$. By \cite[Proposition 6.27]{RoWe98} we obtain  $\lambda\in N_{\TD}(0)\subset N_D(F(\xb))$ and the same arguments as used just before yield \eqref{EqAuxQM_Stat} showing $\Q_M$-stationarity of $\xb$.
\end{proof}

We summarize the relations between the various stationarity concepts in the following picture.
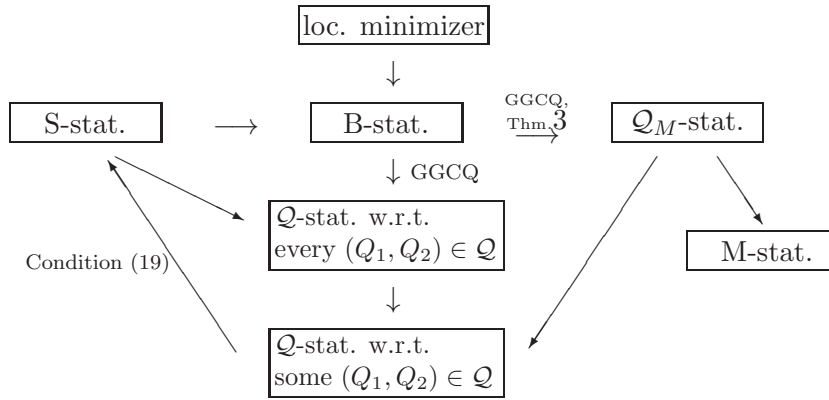
\unitlength1mm
\[\begin{picture}(110,52)
\put(0,34){\framebox(20,5){S-stat.}}
\put(20,36.5){\parbox{20mm}{\[\longrightarrow\]}}
\put(40,34){\framebox(20,5){B-stat.}}
\put(60,36.5){\parbox{20mm}{\[\mathop{\longrightarrow}\limits^{{\rm GGCQ,}\atop{\rm Thm. \ref{ThM_betaStat}}}\]}}
\put(80,34){\framebox(20,5){$\Q_M$-stat.}}
\put(34.5,17){\framebox(31,9){ \parbox{31mm}{\small$\Q$-stat. w.r.t.\\ every $(Q_1,Q_2)\in\Q$}}}
\put(34.5,0){\framebox(31,9){ \parbox{31mm}{\small$\Q$-stat. w.r.t.\\ some $(Q_1,Q_2)\in\Q$}}}
\put(90,17){\framebox(20,5){ M-stat.}}
\put(38.5,47){\framebox(25,5){loc. minimizer}}
\put(50,12){$\downarrow$}
\put(50,29){$\downarrow$ \scriptsize GGCQ}
\put(50,42){$\downarrow$}
\put(30,6){\vector(-2,3){17}}
\put(2,17){\scriptsize Condition \eqref{EqEqualityBetaStat}}
\put(14,32){\vector(2,-1){17}}
\put(86,32){\vector(-2,-3){17}}
\put(94,32){\vector(2,-3){6}}
\end{picture}\]

Below we will work out the concepts of $\Q$- and $\Q_M$-stationarity for the special cases of mathematical programs with complementarity constraints, vanishing constraints and  constraints involving a generalized equation, respectively, and in the first two cases we will present explicit expressions for the pair $(Q_1,Q_2)$ establishing $\Q_M$-stationarity.

Now we consider another possibility to estimate the regular normal cone to $\Omega$, which  is an enhancement of the approach used in the recent paper \cite{GfrOut14a}.
For every nonempty convex cone $Q\subset\R^m$ we define
\[\Tb(Q):=\TD\cap\Big(\big(\range\nabla F(\xb)\cap \TD\big)+Q\Big),\]
i.e. $\Tb(Q)$ is the collection of all $t\in \TD$ such that there are $u\in\R^n$ and $q\in Q$  with
\[\nabla F(\xb)u=t-q\in \TD.\]
Further we define
\[\Cb(Q):=\{u\mv\nabla F(\xb)u\in\co \Tb(Q)\}.\]
It is easy to see that both $\Tb(Q)$ and $\Cb(Q)$ are cones,  that $\Cb(Q)$ is convex and that $\TD\cap Q\subset \Tb(Q)$.
\begin{theorem}\label{ThRegNormalConeIncl}For every nonempty convex cone $Q\subset \R^m$ satisfying
\begin{equation}
  \label{EqCondQ}\widehat N_\Omega(\xb)\subset \{u\mv\nabla F(\xb)u\in Q\}^\circ
\end{equation}
there holds
\begin{equation}\label{EqInclRegNormalConeNew}
\widehat N_\Omega(\xb)= (\Cb(Q))^\circ.
\end{equation}
\end{theorem}
\begin{proof}
 We first show the inclusion $\widehat N_\Omega(\xb)\subset(\Cb(Q))^\circ$. Let $x^\ast\in \widehat N_\Omega(\xb)$ be arbitrarily fixed. In order to show $x^\ast \in (\Cb(Q))^\circ$ we have to prove $\skalp{x^\ast,u}\leq 0\ \forall u\in\Cb(Q)$.  Consider any $u\in\Cb(Q)$. Since $\nabla F(\xb)u\in \co \Tb(Q)$, $\nabla F(\xb)u$ can be represented  as convex combination $\sum_{i=1}^N\alpha_it_i$ of elements $t_i\in\Tb(Q)$, $i=1,\ldots,N$ with coefficients $\alpha_i\in[0,1]$, $\sum_{i=1}^N\alpha_i=1$. By the definition of the set $\Tb(Q)$ we can find for each $i=1,\ldots,N$, elements $u_i\in\R^n$ and $q_i\in Q$ such that
\[\nabla F(\xb)u_i=t_i-q_i\in \TD.\]
By taking into account that  $x^\ast\in \widehat N_\Omega(\xb)=\{u\mv\nabla F(\xb)u\in\TD\}^\circ$ by GGCQ, we obtain $\skalp{x^\ast,u_i}\leq 0$ $\forall i$. Further we have
\[\nabla F(\xb)u=\sum_{i=1}^N\alpha_it_i=\sum_{i=1}^N\alpha_i(\nabla F(\xb)u_i+q_i)\] and therefore
\[\nabla F(\xb)(u-\sum_{i=1}^N\alpha_iu_i)=\sum_{i=1}^N\alpha_iq_i.\]
Since $Q$ is assumed to be convex, we conclude $\nabla F(\xb)(u-\sum_{i=1}^N\alpha_iu_i)\in Q$ and hence, by using \eqref{EqCondQ}, we can argue $\skalp{x^\ast, u-\sum_{i=1}^N\alpha_iu_i}\leq 0$. This yields
\[\skalp{x^\ast,u}\leq\skalp{x^\ast,\sum_{i=1}^N\alpha_iu_i}=\sum_{i=1}^N\alpha_i\skalp{x^\ast,u_i}\leq0\]
and, since $u\in \Cb(Q)$ was arbitrary, we derive the  claimed inclusion $x^\ast\in (\Cb(Q))^\circ$.
In order to show the reverse inclusion $\widehat N_\Omega(\xb)\supset(\Cb(Q))^\circ$ consider $x^\ast\in (\Cb(Q))^\circ$. Then for arbitrary $u\in\Tlin(\xb)$ we have
\[t:=\nabla F(\xb)u = t-0\in T_D(F(\xb)),\]
showing $t\in \Tb(Q)$ and $u\in \Cb(Q)$. Hence, $\skalp{x^\ast,u}\leq 0$ and, because $u\in\Tlin(\xb)$ was chosen arbitrarily, we conclude $x^\ast\in (\Tlin(\xb))^\circ=\widehat N_\Omega(\xb)$ by GGCQ, and $(\Cb(Q))^\circ\subset\widehat N_\Omega(\xb)$ follows.
\end{proof}
\begin{remark}
  \label{RemLargeQ}Condition \eqref{EqCondQ} is in particular fulfilled, if $Q\subset\TD$.
\end{remark}
Of course, in practice it is a difficult task to compute $(\Cb(Q))^\circ$.
In practical applications, for given $Q$ we try to find a cone $\tilde{\cal T}\subset \Tb(Q)$ and then apply Proposition \ref{PropPolarCones} to obtain
\begin{equation}\label{EqPolarCone}(\Cb(Q))^\circ\subset\{u\mv\nabla F(\xb)u\in\co \tilde {\cal T}\}^\circ=\nabla F(\xb)^T\tilde{\cal T}^\circ,
\end{equation}
provided there exists some $u$ with $\nabla F(\xb)u\in \ri\co \tilde{\cal T}$ or $\tilde{\cal T}$ is  polyhedral.
Using \eqref{EqPolarCone} we obtain the following corollary from Theorem \ref{ThRegNormalConeIncl}.
\begin{corollary}\label{CorS_Stat}
Assume that there exists some convex cone $Q\subset \R^m$ fulfilling \eqref{EqCondQ} and some cone $\tilde{\cal T}\subset \Tb(Q)$ such that $\NrD=\tilde{\cal T}^\circ$ and either there is some $u\in\R^n$ with $\nabla F(\xb)u\in\ri \co \tilde{\cal T}$ or $\tilde{\cal T}$ is polyhedral. Then
\[\widehat N_\Omega(\xb)=\nabla F(\xb)^T\NrD.\]
\end{corollary}
\begin{proof}
  Using \eqref{EqInclRegNormalCone},  Theorem \ref{ThRegNormalConeIncl} and \eqref{EqPolarCone} together with  the assumptions of the corollary we obtain
  \begin{eqnarray*}\nabla F(\xb)^T \NrD\subset\widehat N_\Omega(\xb)= (\Cb(Q))^\circ\subset \nabla F(\xb)^T\tilde{\cal T}^\circ
  =\nabla F(\xb)^T\NrD
  \end{eqnarray*}
  and the assertion follows.
\end{proof}
\section{Application to MPCC}
In this section we consider a {\em mathematical program with complementarity constraints} (MPCC) of the form
\begin{eqnarray}
\nonumber\min &&f(x)\\
\nonumber\text{subject to }&&h(x)=0,\\
\label{EqMPCC}&&g(x)\leq0,\\
\nonumber&&0\leq G(x)\perp H(x)\geq 0,
\end{eqnarray}
where $f:\R^n\to\R$, $h:\R^n\to\R^{m_E}$, $g:\R^n\to\R^{m_I}$, $G:\R^n\to\R^{m_C}$ and $H:\R^n\to\R^{m_C}$ are assumed to be  continuously differentiable. There are several possibilities to write the constraints of \eqref{EqMPCC} in the form \eqref{EqConstrSyst}, we use here the formulation with
\[F(x)=\big(h(x),g(x),-G_1(x),-H_1(x),\ldots,-G_{m_C}(x),-H_{m_C}(x)\big),\quad D=\{0\}^{m_E}\times\R_-^{m_I}\times D_C^{m_C},\]
where
\[D_C:=\{(a,b)\in\R^2_-\mv ab=0\}.\]
In what follows we denote the feasible set of \eqref{EqMPCC} by $\Omega_C$. Given a feasible point  $\xb\in\Omega_C$ we introduce the following index sets of constraints active at $\xb$:
\begin{eqnarray*}
I^g&:=&\{i\in \{1,\ldots, m_I\}\mv g_i(\xb)=0\},\\
I^{0+}&:=&\{i\in \{1,\ldots, m_C\}\mv G_i(\xb)=0 <H_i(\xb)\},\\
I^{00}&:=&\{i\in \{1,\ldots, m_C\}\mv G_i(\xb)=0 =H_i(\xb)\},\\
I^{+0}&:=&\{i\in \{1,\ldots, m_C\}\mv G_i(\xb)>0 =H_i(\xb)\}.
\end{eqnarray*}
Straightforward calculations yield that
\begin{equation}\label{EqT_D_MPCC}T_D(F(\xb))=\{0\}^{m_E}\times T_{\R^{m_I}_-}(g(\xb))\times\prod_{i=1}^{m_C}T_{D_C}(-G_i(\xb), -H_i(\xb))
\end{equation}
with $T_{\R^{m_I}_-}(g(\xb))=\{v\in\R^{m_I}\mv v_i\leq 0, i\in I^g\}$,
\[T_{D_C}(-G_i(\xb), -H_i(\xb))=\begin{cases}\{0\}\times\R&\text{if }i\in I^{0+},\\
D_C&\text{if }i\in I^{00},\\
\R\times \{0\}&\text{if }i\in I^{+0},\end{cases}
\]
and consequently $T^{\rm lin}_{\Omega_C}(\xb)$ is the collection of all $u\in\R^n$ fulfilling the system
\begin{eqnarray}\nonumber&&\nabla h(\xb)u= 0,\\
\nonumber &&\nabla g_i(\xb)u\leq 0,\ i\in I^g\\
\label{EqLinTanConeMPCC} &&-\nabla G_i(\xb)u=0,\ i\in I^{0+}\\
\nonumber &&-\nabla H_i(\xb)u= 0,\ i\in I^{+0},\\
\nonumber && 0\geq-\nabla G_i(\xb)u\perp -\nabla H_i(\xb)u\leq 0,\ i\in I^{00}.
 \end{eqnarray}
 Further we have
  \[\quad
\widehat N_{D_C}(-G_i(\xb), -H_i(\xb))=\begin{cases}\R\times\{0\}&\text{if }i\in I^{0+},\\
\R_+\times\R_+&\text{if }i\in I^{00},\\
 \{0\}\times\R&\text{if }i\in I^{+0},\end{cases}\]
 $N_{D_C}(-G_i(\xb), -H_i(\xb))=\widehat N_{D_C}(-G_i(\xb), -H_i(\xb))$ for $i\in I^{+0}\cup I^{0+}$ and $N_{D_C}(-G_i(\xb), -H_i(\xb))=(\R_+\times\R_+)\cup(\{0\}\times\R)\cup(\R\times \{0\})$ for $i\in I^{00}$; cf. \cite{ FleKan06, FleKanOut07, Ye99}.

 Note that GACQ for MPCC is equivalent to MPEC-ACQ as introduced by Flegel and Kanzow \cite{FleKan05}. Similarly, GGCQ for MPCC ie equivalent to MPEC-GCQ \cite{FleKan06}.

In order to apply Theorem \ref{ThQ} and the concept of $\Q$-stationarity we define for every partition $(\beta_1,\beta_2)$ of the biactive index set $I^{00}$ the
convex polyhedric cone
\[Q^{\beta_1,\beta_2}_{CC}:=\{0\}^{m_E}\times T_{\R^{m_I}_-}(g(\xb))\times\prod_{i=1}^{m_C}\tau^{\beta_1,\beta_2}_{i},\]
where $\tau^{\beta_1,\beta_2}_{i}:=T_{D_C}(-G_i(\xb), -H_i(\xb))$ if $i\in I^{0+}\cup I^{+0}$ and
\[\tau^{\beta_1,\beta_2}_{i}:=\begin{cases}
\{0\}\times\R_-&\mbox{if $i\in\beta_1$,}\\
\R_-\times\{0\}&\mbox{if $i\in\beta_2$.}
\end{cases}
\]
\begin{lemma}\label{LemBasPropMPCC}
  For every partition $(\beta_1,\beta_2)\in\Pa(I^{00})$ the pair $(Q_1,Q_2)=(Q^{\beta_1,\beta_2}_{CC},Q^{\beta_2,\beta_1}_{CC})$ consists of two closed convex cones fulfilling \eqref{EqCondQi} and \eqref{EqNecEqual}.
\end{lemma}
\begin{proof}
  It is easy to see that both cones $Q_j$, $j=1,2$ are closed convex polyhedral cones fulfilling $Q_j\subset \TD$ and by using Proposition \ref{PropPolarCones} we conclude that $\big(\nabla F(\xb)^{-1}Q_j\big)^\circ=\nabla F(\xb)^T(Q_j)^\circ$. There remains to show that $(Q^{\beta_1,\beta_2}_{CC})^\circ\cap(Q^{\beta_2,\beta_1}_{CC})^\circ=\NrD$. Since for every $i\in I^{00}=\beta_1\cup\beta_2$ we have $\tau^{\beta_1,\beta_2}_i\cup \tau^{\beta_2,\beta_1}_i=  (\{0\}\times\R_-)\cup (\R_-\times\{0\})=D_C=T_{D_C}(-G_i(\xb), -H_i(\xb))$ and for every $i\in I^{0+}\cup I^{+0}$ we have $\tau^{\beta_1,\beta_2}_i\cup \tau^{\beta_2,\beta_1}_i=T_{D_C}(-G_i(\xb), -H_i(\xb))$ by the definition, we obtain from \eqref{EqPolarConeProd} that
  \begin{eqnarray*}(Q^{\beta_1,\beta_2}_{CC})^\circ\cap(Q^{\beta_2,\beta_1}_{CC})^\circ&=&\left(\{0\}^{m_E}\right)^\circ\times \left(T_{\R^{m_I}_-}(g(\xb))\right)^\circ\times\prod_{i=1}^{m_C}\left(\tau^{\beta_1,\beta_2}_i\cup \tau^{\beta_2,\beta_1}_i\right)^\circ\\
  &=&\left(\{0\}^{m_E}\right)^\circ\times \left(T_{\R^{m_I}_-}(g(\xb))\right)^\circ\times\prod_{i=1}^{m_C}\left(T_{D_C}(-G_i(\xb), -H_i(\xb))\right)^\circ=\NrD\end{eqnarray*}
  and the lemma is proved.
\end{proof}
It is easy to see that $\TD$ is the union taken over all partitions $(\beta_1,\beta_2)\in \Pa(I^{00})$ of the cones $Q^{\beta_1,\beta_2}_{CC}$ and  therefore
$\NrD=\bigcap_{(\beta_1,\beta_2)\in\Pa(I^{00})}\big(Q_{CC}^{\beta_1,\beta_2}\big)^\circ$. We have shown in Lemma \ref{LemBasPropMPCC} that this intersection of $2^{\vert I^{00}\vert}$ many polar cones can be replaced by the intersection of two polar cones $(Q^{\beta_1,\beta_2}_{CC})^\circ\cap(Q^{\beta_2,\beta_1}_{CC})^\circ$. Since
\[\Tlin(\xb)=\bigcup_{(\beta_1,\beta_2)\in\Pa(I^{00})}\nabla F(\xb)^{-1} Q_{CC}^{\beta_1,\beta_2}\]
and under the assumption of GGCQ
\[\widehat N_{\Omega}(\xb)=(\Tlin(\xb))^\circ=\bigcap_{(\beta_1,\beta_2)\in\Pa(I^{00})}\big(\nabla F(\xb)^{-1} Q_{CC}^{\beta_1,\beta_2}\big)^\circ=\bigcap_{(\beta_1,\beta_2)\in\Pa(I^{00})}\nabla F(\xb)^T \big(Q_{CC}^{\beta_1,\beta_2}\big)^\circ,\]
we expect that the replacement of the intersection of the $2^{\vert I^{00}\vert}$ many cones $\nabla F(\xb)^T \big(Q_{CC}^{\beta_1,\beta_2}\big)^\circ$ by the intersection $(Q^{\beta_1,\beta_2}_{CC})^\circ\cap(Q^{\beta_2,\beta_1}_{CC})^\circ$ of two cones can result in a tight inclusion which can be even exact under some reasonable assumptions.

Note that
\begin{eqnarray*}\nabla F(\xb)^{-1} Q_{CC}^{\beta_1,\beta_2}=\{u\in\R^n\mv \begin{array}[t]{ll}\nabla h(\xb)u= 0,& \nabla g_i(\xb)u\leq0,\ i\in I^g,\\
-\nabla G_i(\xb)u=0,\ i\in I^{0+}\cup\beta_1,& -\nabla H_i(\xb)u\leq 0,\ i\in\beta_1,\\
-\nabla G_i(\xb)u\leq 0,\ i\in \beta_2,&-\nabla H_i(\xb)u= 0,\ i\in I^{+0}\cup\beta_2\}.
\end{array}
\end{eqnarray*}

In the sequel we will use the  sets of multipliers
\begin{eqnarray*}
  \label{EqR_CC}{\cal R}_{CC}&:=&\{(\mu^h,\mu^g,\mu^G,\mu^H)\in\R^{m_E}\times \R^{m_I}\times\R^{m_C}\times\R^{m_C}\mv\\
  &&\qquad\mu_i^g=0,\ i\in\{1,\ldots,m_I\}\setminus I^g,\quad \mu^G_i=0,\ i\in I^{+0},\quad \mu^H_i=0,\ i\in I^{0+} \}
  \end{eqnarray*}
  and
  \begin{eqnarray*}
  \label{EqN_CC}{\cal N}_{CC}&:=&\ker F(\xb)^T\cap {\cal R}_{CC}\\
  \nonumber &=&\quad\{(\mu^h,\mu^g,\mu^G,\mu^H)\in{\cal R}_{CC}\mv\\
  &&\qquad
  \sum_{i=1}^{m_E}\mu_i^h\nabla h_i(\xb)+ \sum_{i=1}^{m_I}\mu_i^g\nabla g_i(\xb)-\sum_{i=1}^{m_C}(\mu_i^G\nabla G_i(\xb)+\mu_i^H\nabla H_i(\xb))=0 \}.
  \end{eqnarray*}
Note that
\begin{equation}\label{EqRegNormMPCC}\NrD=\{(\lambda^h,\lambda^g,\lambda^G,\lambda^H)\in {\cal R}_{CC}\mv \lambda_i^g\geq 0,\ i\in I^g,\ \lambda^G_i\geq 0,\lambda_i^H\geq 0,\ i\in I^{00}\}\end{equation}
and
\begin{equation}
  \label{EqLimNormMPCC}
  N_D(F(\xb))=\{(\lambda^h,\lambda^g,\lambda^G,\lambda^H)\in {\cal R}_{CC}\mv \lambda_i^g\geq 0,\ i\in I^g,\ \lambda^G_i>0,\lambda_i^H> 0\mbox{ or } \lambda^G_i\lambda^H_i=0,\ i\in I^{00}\}.
\end{equation}
We now apply Theorem \ref{ThQ} to estimate the regular normal cone $\widehat N_{\Omega_C}(\xb)$ of the MPCC \eqref{EqMPCC}.
\begin{proposition}
  \label{PropMPCC}
  Let $\xb$ belong to the feasible region  $\Omega_C$ of the MPCC \eqref{EqMPCC} and assume that GGCQ is fulfilled at $\xb$.  Then for every partition $(\beta_1,\beta_2)$ of the index set $I^{00}$ we have
  \begin{eqnarray}
    \nonumber\widehat N_{\Omega_C}(\xb)&\subset& \big\{\sum_{i=1}^{m_E}\lambda_i^h\nabla h_i(\xb)+ \sum_{i=1}^{m_I}\lambda_i^g\nabla g_i(\xb)-\sum_{i=1}^{m_C}(\lambda_i^G\nabla G_i(\xb)+\lambda_i^H\nabla H_i(\xb))\mv\\
    \label{EqInclN_MPCC}&&\qquad\qquad (\lambda^h,\lambda^g,\lambda^G,\lambda^H)\in \tilde N^{\beta_1,\beta_2}_{CC}\big\}=:M^{\beta_1,\beta_2}_{CC},
  \end{eqnarray}
  where
  \begin{eqnarray}
  \nonumber\tilde N^{\beta_1,\beta_2}_{CC}
  &:=&\{(\lambda^h,\lambda^g,\lambda^G,\lambda^H)\in{\cal R}_{CC}\mv
    \begin{array}[t]{ll}\exists (\mu^h,\mu^g,\mu^G,\mu^H)\in{\cal N}_{CC}:&
    \lambda^g_i\geq\max\{\mu^g_i,0\},\ i\in I^g,\\
    &\lambda^G_i\geq\mu^G_i,\lambda^H_i\geq 0,\ i\in\beta_1,\\
    &\lambda^G_i\geq 0,\lambda^H_i\geq\mu^H_i,\ i\in\beta_2\}
    \end{array}\\
  \label{EqN_beta_CC}    &=&(Q^{\beta_1,\beta_2}_{CC})^\circ\cap(\ker\nabla F(\xb)^T+(Q^{\beta_2,\beta_1}_{CC})^\circ).
  \end{eqnarray}
\end{proposition}
\begin{proof}We apply \eqref{EqInclRegNormalConeQ1_Q2} with $(Q_1,Q_2)=(Q^{\beta_1,\beta_2}_{CC},Q^{\beta_2,\beta_1}_{CC})$. All we have to show is the equation \eqref{EqN_beta_CC}.
Obviously we have $(Q^{\beta_1,\beta_2}_{CC})^\circ=\R^{m_E}\times N_{\R^{m_I}_-}(g(\xb))\times\prod_{i=1}^{m_C}(\tau^{\beta_1,\beta_2}_{i})^\circ$ and the set $(Q^{\beta_1,\beta_2}_{CC})^\circ\cap(\ker\nabla F(\xb)^T+(Q^{\beta_2,\beta_1}_{CC})^\circ)$ consists of all $\lambda=(\lambda^h,\lambda^g, \lambda^G,\lambda^H)$ such that there exists $\eta=(\eta^h,\eta^g,\eta^G,\eta^H)\in (Q^{\beta_2,\beta_1}_{CC})^\circ$ and some $\mu=(\mu^h,\mu^g,\mu^G,\mu^H)\in\ker\nabla F(\xb)^T$ such that
\begin{eqnarray*}
  \lambda=\eta +\mu\in (Q^{\beta_1,\beta_2}_{CC})^\circ.
\end{eqnarray*}
We proceed with an analysis of the different cases:
\begin{enumerate}
\item{Equality constraints:} We obtain $\lambda^h=\eta^h+\mu^h\in\R^{m_E}$, $\mu^h\in\R^{m_E}$, $\eta^h\in\R^{m_E}$,
i.e., $\lambda^h,\mu^h\in\R^{m_E}$.
\item{Inequality constraints:} For $i\in I^g$ we have $\lambda^g_i=\eta^g_i+ \mu^g_i\geq 0$,  $\eta^g_i\geq 0$ or equivalently $\lambda^g_i\geq\max\{0,\mu^g_i\}$, whereas for $i\in\{1,\ldots,m_I\}\setminus I^g$ we obtain $\lambda^g_i=\eta^g_i=0$ which yields $\mu^g_i=0$.
\item{$i\in I^{0+}$:} Since $(\tau^{\beta_1,\beta_2}_i)^\circ=(\tau^{\beta_2,\beta_1}_i)^\circ=\R\times\{0\}$, we obtain $\lambda^H_i=\eta^H_i=0$ and consequently also $\mu^H_i=0$.
\item{$i\in I^{+0}$:} Similarly as in the previous case we obtain $\lambda^G_i=\mu^G_i=0$.
\item{$i\in\beta_1$:} Since $(\tau^{\beta_1,\beta_2}_i)^\circ=\R\times \R_+$, $(\tau^{\beta_2,\beta_1}_i)^\circ=\R_+\times\R$ we have
\[(\lambda^G_i,\lambda^H_i)=(\eta^G_i,\eta^H_i)+(\mu^G_i, \mu^H_i)\in\R\times\R_+,\]
and $(\eta^G_i,\eta^H_i)\in\R_+\times\R$. This can be written equivalently as $\lambda^G_i\geq\mu^G_i$, $\lambda^H_i\geq 0$.
\item{$i\in\beta_2$:} Similarly as in the previous case we obtain $\lambda^G_i\geq 0$, $\lambda^H_i\geq\mu^H_i$.
\end{enumerate}
We see that $\tilde N^{\beta_1,\beta_2}_{CC}=(Q^{\beta_1,\beta_2}_{CC})^\circ\cap(\ker\nabla F(\xb)^T+(Q^{\beta_2,\beta_1}_{CC})^\circ)$ and the claimed result follows from \eqref{EqInclRegNormalConeQ1_Q2}.
\end{proof}

\begin{theorem}\label{ThMPCCRegNormalCone}
    Let $\xb$ belong to the feasible region  $\Omega_C$ of the MPCC \eqref{EqMPCC} and assume that GGCQ is fulfilled at $\xb$. Further assume that there is some partition $(\beta_1,\beta_2)$ of the index set $I^{00}$ such that
    for every $\mu\in {\cal N}_{CC}$ we have
    \begin{eqnarray*}&&\mu_i^G\mu_{i'}^G\geq 0, \mu_i^H\mu_{i'}^H\geq 0\ \forall (i,i')\in\beta_1\times\beta_2,\\
     &&\mu_i^G\mu_{i'}^H\geq 0\ \forall (i,i')\in\beta_1\times\beta_1,\\
     &&\mu_i^G\mu_{i'}^H\geq 0\ \forall (i,i')\in\beta_2\times\beta_2.
    \end{eqnarray*}
    Then
    \begin{eqnarray*}\widehat N_{\Omega_C}(\xb)&=&M^{\beta_1,\beta_2}_{CC}=\nabla F(\xb)^T\NrD.
    \end{eqnarray*}
\end{theorem}
\begin{proof}Due to \eqref{EqN_beta_CC}, \eqref{EqInclN_MPCC} and Theorem \ref{ThQ} we only have to show that \eqref{EqEqualityBetaStat}, i.e.
 \[M^{\beta_1,\beta_2}_{CC}\subset \nabla F(\xb)^T\NrD,\]
 holds.
Consider $x^\ast\in M^{\beta_1,\beta_2}_{CC}$. Then we have the representation
\begin{eqnarray*}x^\ast &=&\sum_{i=1}^{m_E}\lambda_i^h\nabla h_i(\xb)+ \sum_{i=1}^{m_I}\lambda_i^g\nabla g_i(\xb)-\sum_{i=1}^{m_C}(\lambda_i^G\nabla G_i(\xb)+\lambda_i^H\nabla H_i(\xb))
\end{eqnarray*}
with $(\lambda^h,\lambda^g,\lambda^G,\lambda^H)\in \tilde N^{\beta_1,\beta_2}_{CC}$.
If $\lambda^G_i\geq 0$ for every $i\in \beta_1$ and $\lambda^H_i\geq 0$ for every $i\in \beta_2$, then the claimed inclusion $x^\ast\in  \nabla F(\xb)^T\NrD$ follows from \eqref{EqRegNormMPCC}. Otherwise, either there is some $j\in\beta_1$ such that $\lambda^G_j< 0$ or some $j\in\beta_2$ such that $\lambda^H_j<0$. We consider first the case when $\lambda^G_j< 0$ for some $j\in\beta_1$. Take the element $(\mu^h,\mu^g,\mu^G,\mu^H)\in{\cal N}_{CC}$ associated with $(\lambda^h,\lambda^g,\lambda^G,\lambda^H)$ according to \eqref{EqN_beta_CC} and set $(\tilde\lambda^h,\tilde\lambda^g,\tilde\lambda^G,\tilde\lambda^H):=(\lambda^h-\mu^h,\lambda^g-\mu^g,\lambda^G-\mu^G,\lambda^H-\mu^H)$.
Then
\begin{eqnarray*}x^\ast &=&\sum_{i=1}^{m_E}\tilde\lambda_i^h\nabla h_i(\xb)+ \sum_{i=1}^{m_I}\tilde\lambda_i^g\nabla g_i(\xb)-\sum_{i=1}^{m_C}(\tilde\lambda_i^G\nabla G_i(\xb)+\tilde\lambda_i^H\nabla H_i(\xb))
\end{eqnarray*}
and
\begin{eqnarray*}
 \tilde\lambda^g_i\geq 0,\ i\in I^g,\ \tilde\lambda^G_i\geq 0,\ i\in\beta_1,\quad \tilde\lambda^H_i\geq 0,\ i\in\beta_2
\end{eqnarray*}
by virtue of \eqref{EqN_beta_CC}. Further, since $0>\lambda^G_j\geq\mu^G_j$ we deduce by the assumptions of the theorem that
$\mu^G_i\leq 0\ \forall i\in \beta_2$, $\mu^H_i\leq 0\ \forall i\in \beta_1$ and consequently
$\tilde\lambda^G_i=\lambda^G_i-\mu^G_i\geq \lambda^G_i\geq 0\ \forall i\in \beta_2$, $\tilde\lambda^H_i=\lambda^H_i-\mu^H_i\geq \lambda^H_i\geq 0\ \forall i\in \beta_1$. Therefore $\tilde\lambda^G_i\geq 0$ and $\tilde\lambda^H_i\geq 0$ holds for every $i\in\beta_1\cup\beta_2$ and
$x^\ast=\nabla F(\xb)^T\NrD$ follows. Similar arguments can be applied in the alternative situation when there exists some $j\in \beta_2$ with $\lambda^H_j<0$.
\end{proof}
Let us compare our approach with the results of Pang and Fukushima \cite{PaFu99}. In \cite{PaFu99} the authors try to detect certain redundancies in the description of the linearized tangent cone and then analyze an equivalent representation of the linearized cone. In this paper we treat only so-called (non)singular inequalities, a more general approach goes beyond the scope of this work.

Given a linear system
\[Ax\leq b,\quad Cx=d\]
an inequality $a_ix \leq b_i$ is said to be {\em nonsingular} if there exists a feasible solution of
this system which satisfies this inequality strictly. Here $a_i$ denotes the i-th row of
the matrix $A$. An inequality is called {\em singular} if it is not nonsingular.

Let us denote by $T^{\rm lin}_{\Omega_C,R}(\xb)$ the set of all $u$ fulfilling the linear system
\begin{eqnarray}\nonumber&&\nabla h(\xb)u= 0,\\
\nonumber &&\nabla g_i(\xb)u\leq 0,\ i\in I^g\\
\label{EqRelLinTanConeMPCC} &&-\nabla G_i(\xb)u=0,\ i\in I^{0+}\\
\nonumber &&-\nabla H_i(\xb)u= 0,\ i\in I^{+0},\\
\nonumber && 0\geq-\nabla G_i(\xb)u,\  -\nabla H_i(\xb)u\leq 0,\ i\in I^{00}.
 \end{eqnarray}
 which is obtained from \eqref{EqLinTanConeMPCC} by relaxing the complementarity condition. Obviously we have $T^{\rm lin}_{\Omega_C}(\xb)\subset T^{\rm lin}_{\Omega_C,R}(\xb)$.

Now let $\beta^G$ denote the set
consisting of all indices $i\in I^{00}$ such that the inequality $-\nabla G_i(\xb)u\leq 0$ is nonsingular in
the system \eqref{EqRelLinTanConeMPCC}. Similarly, we denote by $\beta^H$ the nonsingular set pertaining
to the inequalities $-\nabla H_i(\xb)u\leq 0$. For notational convenience we introduce also the set $\beta^{GH}:=\beta^G\cap\beta^H$.

Using the set $\beta^{GH}$ we arrive at the following description of the linearized cone:
\begin{equation}\label{EqRedLinTanConeMPCC}T^{\rm lin}_{\Omega_C}(\xb)=\{u\in\R^n\mv\begin{array}[t]{l}\nabla h(\xb)u= 0,\\
\nabla g_i(\xb)u\leq 0,\ i\in I^g,\\
-\nabla G_i(\xb)u=0,\ i\in I^{0+},\\
-\nabla H_i(\xb)u= 0,\ i\in I^{+0},\\
0\geq-\nabla G_i(\xb)u,\  -\nabla H_i(\xb)u\leq 0,\ i\in I^{00}\setminus\beta^{GH},\\
0\geq-\nabla G_i(\xb)u\perp  -\nabla H_i(\xb)u\leq 0,\ i\in \beta^{GH}\}.
\end{array}
\end{equation}
This can be seen from the fact that every $u$ belonging to the set on the right hand side of \eqref{EqRedLinTanConeMPCC} also belongs to $T^{\rm lin}_{\Omega_C,R}(\xb)$ and therefore for every $i\in
I^{00}\setminus\beta^{GH}=(I^{00}\setminus \beta^G)\cup(I^{00}\setminus \beta^H)$ either the inequality $-\nabla G_i(\xb)u\leq 0$ or the inequality $-\nabla H_i(\xb)u\leq 0$ is singular and consequently fulfilled with equality, implying that complementarity holds. Now the representation \eqref{EqRedLinTanConeMPCC} of the linearized cone  has the same structure as the original representation \eqref{EqLinTanConeMPCC} and we can apply Theorem \ref{ThMPCCRegNormalCone} to \eqref{EqRedLinTanConeMPCC} in order to obtain the following corollary.

\begin{corollary}\label{CorMPCCPang}
    Let $\xb$ belong to the feasible region  $\Omega_C$ of the MPCC \eqref{EqMPCC} and assume that GGCQ is fulfilled at $\xb$. Further assume that there is some partition $(\beta^{GH}_1,\beta^{GH}_2)$ of the index set $\beta^{GH}$ such that for every $\mu\in {\cal N}_{CC}$ there holds
    \begin{eqnarray}
\nonumber &&\mu_i^G\mu_{i'}^G\geq 0, \mu_i^H\mu_{i'}^H\geq 0\ \forall (i,i')\in\beta^{GH}_1\times\beta^{GH}_2,\\
\label{EqWeakenedA3}&&\mu_i^G\mu_{i'}^H\geq 0\ \forall (i,i')\in\beta^{GH}_1\times\beta^{GH}_1,\\
\nonumber&&\mu_i^G\mu_{i'}^H\geq 0\ \forall (i,i')\in\beta^{GH}_2\times\beta^{GH}_2.
    \end{eqnarray}
    Then
    \begin{eqnarray*}\widehat N_{\Omega_C}(\xb)&=&\nabla F(\xb)^T\NrD.
    \end{eqnarray*}
\end{corollary}
 \begin{proof}The representation \eqref{EqRedLinTanConeMPCC} has the form $T^{\rm lin}_{\Omega_C}(\xb)=\{u\in\R^n\mv\nabla F(\xb)u\in T^{GH}\}$ with
\[T^{GH}=\{0\}^{m_E}\times T_{\R^{m_I}_-}(g(\xb))\times\prod_{i=1}^{m_C} \tilde T^{GH}_i,\ \tilde T^{GH}_i=\begin{cases}
  \R^2_-&\mbox{if $i\in I^{00}\setminus\beta^{GH}$,}\\
  T_{D_C}(-G_i(\xb), -H_i(\xb))&\mbox{if $i\in I^{0+}\cup I^{+0}\cup\beta^{GH}$}
\end{cases}\]
and from Theorem \ref{ThMPCCRegNormalCone} we obtain $\widehat N_{\Omega_C}(\xb)=\nabla F(\xb)^T(T^{GH})^\circ$. It is easy to see that $(T^{GH})^\circ=(\TD)^\circ=\NrD$ and thus the assertion follows.
\end{proof}

The statement of Corollary \ref{CorMPCCPang} was shown in \cite[Theorem 2]{PaFu99} under the assumption (A3), which reads in our notation that there exists a partition $(\beta^{GH}_1,\beta^{GH}_2)$ of the index set $\beta^{GH}$ such that for every $\mu\in{\cal N}_{CC}$ one has
    \begin{eqnarray}
    \nonumber &&\mu_i^G\mu_{i'}^G\geq 0\ \forall (i,i')\in\beta^{GH}_1\times(\beta^G\setminus\beta^{GH}_1),\\
    \label{EqA3}     &&\mu_i^G\mu_{i'}^H\geq 0\  \forall (i,i')\in\beta^{GH}_1\times(\beta^H\setminus\beta^{GH}_2),\\
    \nonumber    &&\mu_i^H\mu_{i'}^H\geq 0\  \forall (i,i')\in\beta^{GH}_2\times(\beta^H\setminus\beta^{GH}_2),\\
    \nonumber     &&\mu_i^G\mu_{i'}^H\geq 0\ \forall (i,i')\in(\beta^G\setminus\beta^{GH}_1)\times\beta^{GH}_2.
    \end{eqnarray}
    Since $\beta^{GH}_2=\beta^{GH}\setminus\beta^{GH}_1\subset\beta^G\setminus \beta^{GH}_1$ and $\beta^{GH}_1\subset\beta^H\setminus \beta^{GH}_2$, our assumption \eqref{EqWeakenedA3} is not stronger than assumption (A3) used by Pang and Fukushima \cite{PaFu99}. In case when $\beta^G\not=\beta^{GH}$ or $\beta^H\not=\beta^{GH}$ our assumption \eqref{EqWeakenedA3} is actually weaker, as the following example demonstrates.
\begin{example}
  Consider the system
  \begin{eqnarray*}
    &&g_1(x):=-x_3-x_4\leq 0,\\
    && g_2(x):=x_2\leq 0,\\
    &&0\leq G_1(x):=x_1\perp H_1(x):=x_2\geq 0,\\
    &&0\leq G_2(x):=x_1+x_3\perp H_2(x):=x_4\geq 0
  \end{eqnarray*}
  at $\xb=(0,0,0,0)$.  Since all constraint functions are linear, GACQ is fulfilled, cf. also \cite[Theorem 3.2]{FleKan05}, and consequently GGCQ  holds as well.
  It is easy to see that $\beta^G=\{1,2\}$ and $\beta^{GH}=\beta^H=\{2\}$ and therefore condition \eqref{EqWeakenedA3} amounts to
  \begin{eqnarray}\label{EqEx1Aux1}&&-\mu^G_1-\mu^G_2=0,\ \mu^g_2-\mu^H_1=0,\ -\mu^g_1-\mu^G_2=0,\ -\mu^g_1-\mu^H_2=0\\
  \label{EqEx1Aux2}     &\Rightarrow&\left\{\begin{array}
    {ll}\mu_i^G\mu_{i'}^G\geq 0, \mu_i^H\mu_{i'}^H\geq 0& \forall (i,i')\in\beta^{GH}_1\times\beta^{GH}_2\\
     \mu_i^G\mu_{i'}^H\geq 0& \forall (i,i')\in\beta^{GH}_1\times\beta^{GH}_1\\
     \mu_i^G\mu_{i'}^H\geq 0& \forall (i,i')\in\beta^{GH}_2\times\beta^{GH}_2
  \end{array}\right.
    \end{eqnarray}
  Since \eqref{EqEx1Aux1} is equivalent to $\mu^H_1=\mu^g_2$, $\mu^H_2=\mu^G_2=-\mu^G_1=-\mu^g_1$, \eqref{EqEx1Aux2} holds with any of the two partitions $\beta^{GH}_1=\{2\},\beta^{GH}_2=\emptyset$ and $\beta^{GH}_1=\emptyset,\beta^{GH}_2=\{2\}$ and therefore Corollary \ref{CorMPCCPang} is applicable. On the other hand, condition \eqref{EqA3} reads
  as
    \begin{eqnarray}
\nonumber      &&-\mu^G_1-\mu^G_2=0,\ \mu^g_2-\mu^H_1=0,\ -\mu^g_1-\mu^G_2=0,\ -\mu^g_1-\mu^H_2=0\\
\label{EqEx1A3} &\Rightarrow&\left\{\begin{array}{ll}\mu_i^G\mu_{i'}^G\geq 0&\forall (i,i')\in\beta^{GH}_1\times(\{1,2\}\setminus\beta^{GH}_1)\\
     \mu_i^G\mu_{i'}^H\geq 0& \forall (i,i')\in\beta^{GH}_1\times(\{2\}\setminus\beta^{GH}_2)\\
     \mu_i^H\mu_{i'}^H\geq 0& \forall (i,i')\in\beta^{GH}_2\times(\{2\}\setminus\beta^{GH}_2)\\
     \mu_i^G\mu_{i'}^H\geq 0& \forall (i,i')\in(\{1,2\}\setminus\beta^{GH}_1)\times\beta^{GH}_2\end{array}\right.
    \end{eqnarray}
    Taking $(\mu^g_1,\mu^g_2,\mu^G_1,\mu^G_2,\mu^H_1,\mu^H_2)=(1,1,1,-1,1-1)$ we obtain that for the partition $\beta^{GH}_1=\emptyset, \beta^{GH}_2=\{2\}$ the condition $\mu_1^G\mu_2^H\geq 0$ is violated, whereas in case when $\beta^{GH}_1=\{2\}, \beta^{GH}_2=\emptyset$ the inequality $\mu_2^G\mu_1^G\geq0$ fails to hold. Thus \cite[Assumption (A3)]{PaFu99} does not hold for this example and therefore the assumption used in our Corollary \ref{CorMPCCPang} is strictly weaker.
\end{example}
%

We introduce now the following stationarity concepts for MPCC  which correspond to Definition \ref{DefQStat} with $\Q=\Q_{CC}$, where
\[\Q_{CC}:=\{(Q^{\beta_1,\beta_2}_{CC},Q^{\beta_2,\beta_1}_{CC})\mv (\beta_1,\beta_2)\mbox{ is partition of }I^{00}\}.\]
Note that there is a one-to-one correspondence between the sets $(Q_1,Q_2)\in\Q_{CC}$ and partitions $(\beta_1,\beta_2)$ of the biactive index set $I^{00}$
\begin{definition}Let $\xb\in\Omega_C$.
\begin{enumerate}
\item   We say that $\xb$ is {\em $\Q$-stationary for the MPCC \eqref{EqMPCC} with respect to the partition $(\beta_1,\beta_2)$ of the index set $I^{00}$} if
  \[0\in \nabla f(\xb)+M^{\beta_1,\beta_2}_{CC},\]
  where $M^{\beta_1,\beta_2}_{CC}$ is given by \eqref{EqInclN_MPCC}.
\item  We say that $\xb$ is {\em $\Q$-stationary for the MPCC \eqref{EqMPCC} } if it is $\Q$-stationary with respect to some  partition $(\beta_1,\beta_2)$ of the index set $I^{00}$.
\item We say that $\xb$ is {\em $\Q_M$-stationary for the MPCC \eqref{EqMPCC}} if there is some partition $(\beta_1,\beta_2)$ of $I^{00}$ such that
\[0\in\nabla f(\xb)+\nabla F(\xb)^T\left((Q^{\beta_1,\beta_2}_{CC})^\circ\cap(\ker\nabla F(\xb)^T+ (Q^{\beta_2,\beta_1}_{CC})^\circ)\cap N_D(F(\xb))\right).\]
\end{enumerate}
\end{definition}
\begin{theorem}\label{ThMPCC_beta_stat}
  Assume that GGCQ is fulfilled at the point $\xb\in\Omega_C$. If $\xb$ is B-stationary, then $\xb$ is $\Q$-stationary for the MPCC \eqref{EqMPCC} with respect to every partition $(\beta_1,\beta_2)$ of $I^{00}$ and it is also $\Q_M$ stationary. Conversely, if $\xb$ is $\Q$-stationary with respect to a partition $(\beta_1,\beta_2)$ of $I^{00}$, which fulfills also the assumptions of Theorem \ref{ThMPCCRegNormalCone}, then $\xb$ is S-stationary and consequently B-stationary.
\end{theorem}
\begin{proof}
  In view of the definitions of B-stationarity and S-stationarity together with Proposition \ref{PropMPCC} and Theorem \ref{ThMPCCRegNormalCone} there is only to show the assertion about $\Q_M$-stationarity. This follows easily from Theorem \ref{ThM_betaStat}(3.) because $\TD=\bigcup_{(\beta_1,\beta_2)\in\Pa(I^{00})}Q^{\beta_1,\beta_2}_{CC}$ is the union of finitely many convex polyhedral cones  generating the collection $\Q$.
\end{proof}
\begin{remark}
  Given a multiplier $\lambda\in N_D(F(\xb))$ verifying the M-stationarity condition $0\in\nabla f(\xb)+\nabla F(\xb)^T\lambda$ we can use the partition $(\beta_1,\beta_2)\in\Pa(I^{00})$ defined by
  \[\beta_1=\{i\in I^{00}\mv \lambda_i^H\geq0\},\quad \beta_2=\{i\in I^{00}\mv \lambda_i^H<0\}\]
  for testing $\xb$ on $\Q_M$-stationarity, because this choice ensures $\lambda\in \big(Q^{\beta_1,\beta_2}_{CC}\big)^\circ$. The computation of such a multiplier $\lambda$ can be done by means of the algorithm presented in the proof of \cite[Theorem 4.3]{Gfr14a}.
\end{remark}
We see that $\Q$-stationarity is a first order necessary condition for $\xb$ being a local minimizer, provided GGCQ is fullfilled, which is to be considered as a very weak constraint qualification. In order to verify $\Q$-stationarity, only a  system of linear equalities and linear inequalities has to be solved, but the main difference to the usual first-order optimality conditions is, that a second multiplier $\mu$ is involved.

Note that postulating GGCQ in our problem setting is equivalent to MPEC-GCQ as given in \cite{FleKan06}. It was shown in \cite{FleKan06} that under MPEC-GCQ any B-stationary point of MPCC is M-stationary. Theorem \ref{ThMPCC_beta_stat} improves this result by stating that even $\Q_M$-stationarity holds.

Let us now turn our attention to the case when the gradients of the constraints active at the point $\xb$,
\[\nabla h_i(\xb),i\in\{1,\ldots,m_E\},\ \nabla g_i(\xb),i\in I^g,\ \nabla G_i(\xb),\in I^{0+}\cup I^{00},\ \nabla H_i(\xb),i\in I^{+0}\cup I^{00}\]
are linearly independent. This constraint qualification is usually named MPEC-LICQ in the literature. Then we obviously  have ${\cal N}_{CC}=\{0\}$ and therefore the assumptions of Theorem \ref{ThMPCCRegNormalCone} hold. Hence, under MPEC-LICQ $\Q$-stationarity automatically implies  S-stationarity and B-stationarity. This is remarkable because M-stationarity does not have this property: Under MPEC-LICQ an M-stationary point is neither S-stationary nor B-stationary in general.  However, in case when MPEC-LICQ does not hold, there also exist examples where a $\Q$-stationary point is not M-stationary and therefore neither M-stationarity implies $\Q$-stationarity nor vice versa. However, the following example shows that $\Q_M$-stationarity is strictly stronger than M-stationarity.
\begin{example}(cf.\cite[Example 3]{Gfr14a})
Consider the MPCC
\begin{eqnarray*}\min_{x\in\R^3}f(x)&:=&x_1+x_2-2x_3\\
\mbox{subject to }g_1(x)&:=&-x_1-x_3\leq0\\
  g_2(x)&:=&-x_2+x_3\leq0\\
  0\leq G_1(x)&:=&x_1\perp x_2=:H_1(x)\geq 0
\end{eqnarray*}
Then $\xb=(0,0,0)$ is not a local minimizer because for every $\alpha>0$ the point $x^\alpha=(0,\alpha,\alpha)$ is feasible and $f(x^\alpha)=-\alpha<0=f(\xb)$. GACQ is fulfilled because all constraints are linear and the linearized cone amounts to
\begin{eqnarray*}
  -u_1-u_3\leq 0,\ -u_2+u_3\leq 0,\ 0\geq -u_1\perp -u_2\leq 0.
\end{eqnarray*}
 Straightforward calculations yield that $\xb$ is M-stationary and $\lambda=(\lambda^g_1,\lambda^g_2,\lambda^G_1,\lambda^H_1)=(1,3,0,-2)$ is the unique multiplier fulfilling the M-stationarity conditions. However, we will now show that $\xb$ is not $\Q_M$-stationary. Assuming that $\xb$ is $\Q_M$-stationary, by taking $\beta_1=\emptyset$, $\beta_2=\{1\}$, there would exist some $\mu=(\mu^g_1,\mu^g_2,\mu^G_1,\mu^H_1)$ verifying
\begin{eqnarray*}
  &&-\mu^g_1-\mu^G_1=0,\ -\mu^g_2-\mu^H_1=0,\ -\mu^g_1+\mu^g_2=0\\
  &&\mu_g^1\leq \lambda^g_1=1,\ \mu_g^2\leq\lambda^g_2=3,\ \mu^H_1\leq\lambda^H_1=-2.
\end{eqnarray*}
But a solution of this system must fulfill
\[\mu^g_2=\mu_g^1\leq 1,\ \mu^g_2=-\mu^H_1\geq 2\]
which is obviously not possible. On the other hand, if we take $\beta_1=\{1\}$, $\beta_2=\emptyset$ then $\lambda\not\in (Q^{\beta_1,\beta_2}_{CC})^\circ$.
Hence $\xb$ is not $\Q_M$-stationary and we have demonstrated that $\Q_M$-stationarity is a stronger property than M-stationarity.
\end{example}

\section{Application to MPVC}
In this section we consider a {\em mathematical program with vanishing constraints} (MPVC) of the form
\begin{eqnarray}
\nonumber\min &&f(x)\\
\nonumber\text{subject to }&&h(x)=0,\\
\nonumber&&g(x)\leq0,\\
\label{EqMPVC}&&H_i(x)\geq0,\ G_i(x)H_i(x)\leq 0,\ i=1,\ldots, m_V,
\end{eqnarray}
where $f:\R^n\to\R$, $h:\R^n\to\R^{m_E}$, $g:\R^n\to\R^{m_I}$, $G:\R^n\to\R^{m_V}$ and $H:\R^n\to\R^{m_V}$ are assumed to be at least continuously differentiable. To transform the constraints into the format \eqref{EqConstrSyst} we use
\[F(x)=\left(h(x),g(x),-H_1(x),G_1(x),\ldots,-H_{m_V}(x),G_{m_V}(x)\right),\quad D=\{0\}^{m_E}\times\R_-^{m_I}\times D_V^{m_V},\]
where
\[D_V:=\{(a,b)\in\R_-\times\R \mv ab\geq0\}.\]
Now we denote the feasible region of \eqref{EqMPVC} by $\Omega_V$ and we introduce the following index sets of constraints active at a feasible point $\xb\in\Omega_V$:
\begin{eqnarray*}
I^g&:=&\{i\in \{1,\ldots, m_I\}\mv g_i(\xb)=0\},\\
I^{0-}&:=&\{i\in \{1,\ldots, m_V\}\mv H_i(\xb)=0 > G_i(\xb)\},\\
I^{00}&:=&\{i\in \{1,\ldots, m_V\}\mv H_i(\xb)=0 =G_i(\xb)\},\\
I^{0+}&:=&\{i\in \{1,\ldots, m_V\}\mv H_i(\xb)=0 < G_i(\xb)\},\\
I^{+0}&:=&\{i\in \{1,\ldots, m_V\}\mv H_i(\xb)>0 = G_i(\xb)\},\\
I^{+-}&:=&\{i\in \{1,\ldots, m_V\}\mv H_i(\xb)>0 > G_i(\xb)\}.\\
\end{eqnarray*}
Straightforward calculations yield that
\[T_D(F(\xb))=\{0\}^{m_E}\times T_{\R^{m_I}_-}(g(\xb))\times\prod_{i=1}^{m_V}T_{D_V}(-H_i(\xb), G_i(\xb))\]
with $T_{\R^{m_I}_-}(g(\xb))=\{v\in\R^{m_I}\mv v_i\leq 0, i\in I^g\}$,
\[T_{D_V}(-H_i(\xb), G_i(\xb))=\begin{cases}\R_-\times\R&\text{if }i\in I^{0-},\\
D_V&\text{if }i\in I^{00},\\
\{0\}\times\R&\text{if }i\in I^{0+},\\
\R\times \R_-&\text{if }i\in I^{+0},\\
\R\times \R&\text{if }i\in I^{+-},\end{cases}\quad
\widehat N_{D_V}(-H_i(\xb), G_i(\xb))=\begin{cases}\R_+\times\{0\}&\text{if }i\in I^{0-},\\
\R_+\times\{0\}&\text{if }i\in I^{00},\\
\R\times\{0\}&\text{if }i\in I^{0+},\\
\{0\}\times \R_+&\text{if }i\in I^{+0},\\
\{0\}\times \{0\}&\text{if }i\in I^{+-}\end{cases}
\]
and consequently, $T^{\rm lin}_{\Omega_V}(\xb)$ is the collection of all $u\in\R^n$ fulfilling the system
\begin{eqnarray}\nonumber&&\nabla h(\xb)u= 0,\\
\nonumber &&\nabla g_i(\xb)u\leq 0,\ i\in I^g,\\
\label{EqLinTanConeMPVC} &&-\nabla H_i(\xb)u=0,\ i\in I^{0+},\\
\nonumber &&-\nabla H_i(\xb)u\leq 0,\ i\in I^{0-}\cup I^{00},\\
\nonumber && (\nabla G_i(\xb)u)(\nabla H_i(\xb)u)\geq 0,\ i\in I^{00},\\
\nonumber &&\nabla G_i(\xb)u\leq 0,\ i\in I^{+0}.
\end{eqnarray}
Further note that $N_{D_V}(-H_i(\xb), G_i(\xb))=\widehat N_{D_V}(-H_i(\xb), G_i(\xb))$, $i\not\in I^{00}$ and $N_{D_V}(-H_i(\xb), G_i(\xb)) =(\R\times\{0\})\cup(\{0\}\times\R_+)$, $i\in I^{00}$.

Similar to MPCC we define for every partition $(\beta_1,\beta_2)$ of the set $I^{00}$ the
cone
\[Q^{\beta_1,\beta_2}_{VC}:=\{0\}^{m_E}\times T_{\R^{m_I}_-}(g(\xb))\times\prod_{i=1}^{m_V}\tau^{\beta_1,\beta_2}_{i},\]
where $\tau^{\beta_1,\beta_2}_{i}:=T_{D_V}( -H_i(\xb),G(\xb))$ if $i\not\in I^{00}$ and
\[\tau^{\beta_1,\beta_2}_{i}:=\begin{cases}
\{0\}\times\R&\mbox{if $i\in\beta_1$,}\\
\R_-\times\R_-&\mbox{if $i\in\beta_2$,}
\end{cases}
\]
\begin{lemma}
  For every partition $(\beta_1,\beta_2)\in\Pa(I^{00})$ the pair $(Q_1,Q_2)=(Q^{\beta_1,\beta_2}_{VC},Q^{\beta_2,\beta_1}_{VC})$ consists of two closed convex cones fulfilling \eqref{EqCondQi} and \eqref{EqNecEqual}.
\end{lemma}
\begin{proof}
The proof follows the same lines as the proof of Lemma \ref{LemBasPropMPCC} and is therefore omitted.
\end{proof}
Similar to the case of MPCC we have
\[\TD=\bigcup_{(\beta_1,\beta_2)\in \Pa(I^{00})}Q^{\beta_1,\beta_2}_{VC}.\]
Consider the following two sets of multipliers,
\begin{eqnarray*}
  {\cal R}_{VC}&:=&\{(\mu^h,\mu^g,\mu^H,\mu^G)\in\R^{m_E}\times \R^{m_I}\times\R^{m_V}\times\R^{m_V}\mv\\
  &&\qquad \mu_i^g=0,\ i\in\{1,\ldots,m_I\}\setminus I^g,\\
  &&\qquad \mu^H_i=0,\ i\in\{1,\ldots,m_V\}\setminus(I^{0-}\cup I^{00}\cup I^{0+}),\\
  &&\qquad \mu^G_i=0,\ i\in\{1,\ldots,m_V\}\setminus(I^{+0}\cup I^{00}) \}
\end{eqnarray*}
and
\begin{eqnarray*}
{\cal N}_{VC}&:=&\{(\mu^h,\mu^g,\mu^H,\mu^G)\in{\cal R}_{VC}\mv\\
  &&\qquad
  \sum_{i=1}^{m_E}\mu_i^h\nabla h_i(\xb)+ \sum_{i=1}^{m_I}\mu_i^g\nabla g_i(\xb)+\sum_{i=1}^{m_V}(\mu_i^G\nabla G_i(\xb)-\mu_i^H\nabla H_i(\xb))=0 \} .
\end{eqnarray*}
Note that
\begin{equation}\label{EqRegNormMPVC}
\NrD=\{\lambda\in{\cal R}_{VC}\mv\begin{array}[t]{l} \lambda^g_i\geq 0, i\in I^g,\ \lambda^H_i\geq 0,\  i\in I^{0-},\ \lambda^G_i\geq0,\ i\in I^{+0},\\
 \lambda^H_i\geq 0,\ \lambda^G_i=0,\ i\in I^{00}\}\end{array}
\end{equation}
and
\begin{equation}\label{EqLimNormMPVC}
N_D(F(\xb))=\{\lambda\in{\cal R}_{VC}\mv\begin{array}[t]{l} \lambda^g_i\geq 0, i\in I^g,\ \lambda^H_i\geq 0,\  i\in I^{0-},\ \lambda^G_i\geq0,\ i\in I^{+0},\\  \lambda^G_i\geq 0,\ \lambda^H_i\lambda^G_i= 0,\ i\in I^{00}\}.\end{array}
\end{equation}
\begin{proposition}\label{PropMPVC}
   Let $\xb$ belong to the feasible region  $\Omega_V$ of the MPVC \eqref{EqMPVC} and assume that GGCQ is fulfilled at $\xb$.  Then for every partition $(\beta_1,\beta_2)$ of the index set $I^{00}$ we have
  \begin{eqnarray}
    \label{EqInclN_MPVC}\widehat N_{\Omega_V}(\xb)&\subset& \{\sum_{i=1}^{m_E}\lambda_i^h\nabla h_i(\xb)+ \sum_{i=1}^{m_I}\lambda_i^g\nabla g_i(\xb)+\sum_{i=1}^{m_V}(\lambda_i^G\nabla G_i(\xb)-\lambda_i^H\nabla H_i(\xb))\mv\qquad\\
    \nonumber&&\qquad\qquad (\lambda^h,\lambda^g,\lambda^H,\lambda^G)\in \tilde N^{\beta_1,\beta_2}_{VC}\}=:M^{\beta_1,\beta_2}_{VC},
  \end{eqnarray}
  where
  \begin{eqnarray*}
    \tilde N^{\beta_1,\beta_2}_{VC}  &:=&\{(\lambda^h,\lambda^g,\lambda^H,\lambda^G)\in{\cal R}_{VC}\mv
    \exists (\mu^h,\mu^g,\mu^H,\mu^G)\in{\cal N}_{VC}:\\
    \nonumber&&\qquad\lambda^g_i\geq\max\{\mu^g_i,0\},\ i\in I^g,\\
    &&\qquad\lambda^H_i\geq\max\{\mu^H_i,0\},\  i\in I^{0-},\quad  \lambda^G_i\geq\max\{\mu^G_i,0\},\ i\in I^{+0},\\
    &&\qquad\lambda^H_i\geq\mu^H_i,\ \mu^G_i\leq\lambda^G_i=0,\ i\in\beta_1,\quad \lambda^H_i\geq 0,\ \lambda^G_i=\mu^G_i\geq 0,\ i\in\beta_2\}\\
    &=&(Q^{\beta_1,\beta_2}_{VC})^\circ\cap(\ker\nabla F(\xb)^T+(Q^{\beta_2,\beta_1}_{VC})^\circ).
\end{eqnarray*}
\end{proposition}
\begin{proof}We can proceed similarly to the proof of Proposition \ref{PropMPCC}.
 We have $(Q^{\beta_1,\beta_2}_{VC})^\circ=\R^{m_E}\times N_{\R^{m_I}_-}(g(\xb))\times\prod_{i=1}^{m_V}(\tau^{\beta_1,\beta_2}_{i})^\circ$ and the set $\tilde N^{\beta_1,\beta_2}_{VC}=(Q^{\beta_1,\beta_2}_{CC})^\circ\cap(\ker\nabla F(\xb)^T+(Q^{\beta_2,\beta_1}_{CC})^\circ)$ consists of all $\lambda=(\lambda^h,\lambda^g, \lambda^H,\lambda^G)$ such that there exists $\eta=(\eta^h,\eta^g,\eta^H,\eta^G)\in (Q^{\beta_2,\beta_1}_{VC})^\circ$ and some $\mu=(\mu^h,\mu^g,\mu^H,\mu^G)\in\ker\nabla F(\xb)^T$ such that
\begin{eqnarray*}
  \lambda=\eta +\mu\in (Q^{\beta_1,\beta_2}_{VC})^\circ.
\end{eqnarray*}
Similar as in the proof of Proposition \ref{PropMPCC} this yields
\begin{eqnarray*}
\lambda^g_i\geq\max\{\mu^g_i,0\},\ i\in I^g,\quad \lambda^g_i=\mu^g_i=0,\ i\in\{1,\ldots,m_I\}\setminus I^g,\\
\lambda^H_i\geq\max\{\mu^H_i,0\},\ \lambda^G_i=\mu^G_i=0,\ i\in I^{0-},\quad \lambda^G_i=\mu^G_i=0,\ i\in I^{0+},\\
\lambda^H_i=\mu^H_i=0,\ \lambda^G_i\geq\max\{\mu^G_i,0\},\ \ i\in I^{+0},\quad \lambda^H_i=\lambda^G_i=\mu^H_i=\mu^G_i=0,\ i\in I^{+-}.
\end{eqnarray*}
Now consider $i\in\beta_1$. Then $(\tau^{\beta_1,\beta_2}_i)^\circ=\R\times\{0\}$ and $(\tau^{\beta_2,\beta_1}_i)^\circ=\R_+\times\R_+$.
Hence
\[(\lambda^H_i,\lambda^G_i)=(\eta^H_i,\eta^G_i)+(\mu^H_i, \mu^G_i)\in\R\times\{0\}\]
and $(\eta^H_i,\eta^G_i)\in\R_+\times\R_+$, or equivalently
\begin{eqnarray*}
\lambda^H_i\geq\mu^H_i,\ \mu^G_i\leq\lambda^G_i=0,\ i\in\beta_1.
\end{eqnarray*}
In case that $i\in\beta_2$ we have $(\tau^{\beta_1,\beta_2}_i)^\circ=\R_+\times\R_+$ and $(\tau^{\beta_2,\beta_1}_i)^\circ=\R\times\{0\}$,
\[(\lambda^H_i,\lambda^G_i)=(\eta^H_i,\eta^G_i)+(\mu^H_i, \mu^G_i)\in\R_+\times\R_+, \]
and $(\eta^H_i,\eta^G_i)\in\R\times\{0\}$, which is equivalent to
\begin{eqnarray*}
\lambda^H_i\geq 0,\ \lambda^G_i=\mu^G_i\geq 0,\ i\in\beta_2.
\end{eqnarray*}
These arguments show that  $\tilde N^{\beta_1,\beta_2}_{VC}$ has the claimed representation and the assertion follows from  \eqref{EqInclRegNormalConeQ1_Q2}.
\end{proof}
In the following theorem we give a sufficient condition for equality in \eqref{EqInclN_MPVC}.
\begin{theorem}\label{ThMPVCRegNormalCone}  Let $\xb$ belong to the feasible region  $\Omega_V$ of the MPVC \eqref{EqMPVC} and assume that GGCQ is fulfilled at $\xb$. Further assume that  there is a partition $(\beta_1,\beta_2)$ of $I^{00}$ such that
\begin{equation}\label{EqSuffEquMPVC}\left.\begin{array}
  {l}(\mu^h,\mu^g,\mu^H,\mu^G)\in{\cal N}_{VC}\\
  \mu_i^G\leq 0,\ i\in\beta_1,\ \mu_i^G\geq 0,\ i\in\beta_2
\end{array}\right\}\Longrightarrow\mu_i^H\geq 0,\ i\in\beta_1,\ \mu^G_i=0,\ i\in\beta_2.
\end{equation}
Then
\begin{eqnarray*}
  \widehat N_{\Omega_V}(\xb)=M^{\beta_1,\beta_2}_{VC}=\nabla F(\xb)^T \widehat N_D(F(\xb)).
\end{eqnarray*}
\end{theorem}
\begin{proof}
Under the assumption of the theorem we conclude that
\begin{eqnarray*}
    \tilde N^{\beta_1,\beta_2}_{VC}
    &\subset&\{(\lambda^h,\lambda^g,\lambda^H,\lambda^G)\in{\cal R}_{VC}\mv
    \lambda^g_i\geq 0,\ i\in I^g,\\
    &&\qquad\lambda^H_i\geq0,\  i\in I^{0-},\quad  \lambda^G_i\geq0,\ i\in I^{+0},\\
    &&\qquad\lambda^H_i\geq0,\ \lambda^G_i=0,\ i\in\beta_1,\quad \lambda^H_i\geq 0,\ \lambda^G_i= 0,\ i\in\beta_2\}=\NrD.
\end{eqnarray*}
Now the claimed result follows from Theorem \ref{ThQ} together with Proposition \ref{PropMPVC} by taking  $(Q_1,Q_2)=(Q^{\beta_1,\beta_2}_{VC},Q^{\beta_2,\beta_1}_{VC})$.
\end{proof}
Next we establish an equivalent formulation of condition \eqref{EqSuffEquMPVC}.
\begin{lemma}\label{LemEquAss1}Let $(\beta_1,\beta_2)$ be a partition of $I^{00}$. Then the following statements are equivalent:
\begin{enumerate}
  \item[(i)]Condition \eqref{EqSuffEquMPVC} is fulfilled.
  \item[(ii)]For every $j\in\beta^1$ there exists some $z^j$ such that
  \begin{eqnarray}
    \nonumber &&\nabla h(\xb)z^j=0,\\
    \nonumber &&\nabla g_i(\xb)z^j=0,\ i\in I^g,\\
    \label{Eq_z_j}&&\nabla G_i(\xb)z^j=0,\ i\in I^{+0},\\
    \nonumber &&\nabla G_i(\xb)z^j\begin{cases}\geq 0,&i\in \beta_1,\\\leq 0,&i\in\beta_2,\end{cases}\\
    \nonumber &&\nabla H_i(\xb)z^j=0,\ i\in I^{0-}\cup I^{00}\cup I^{0+}\setminus\{j\},\\
    \nonumber &&\nabla H_j(\xb)z^j=-1
  \end{eqnarray}
  and there is some $\bar z$ such that
  \begin{eqnarray}
    \nonumber &&\nabla h(\xb)\bar z=0,\\
    \nonumber &&\nabla g_i(\xb)\bar z=0,\ i\in I^g,\\
    \label{Eq_bar_z}&&\nabla G_i(\xb)\bar z=0,\ i\in I^{+0},\\
    \nonumber &&\nabla G_i(\xb)\bar z\begin{cases}\geq 0,&i\in \beta_1,\\\leq -1&i\in\beta_2,\end{cases}\\
    \nonumber &&\nabla H_i(\xb)\bar z=0,\ i\in I^{0-}\cup I^{00}\cup I^{0+}.
  \end{eqnarray}
\end{enumerate}
\end{lemma}
\begin{proof}
Condition \eqref{EqSuffEquMPVC} is fulfilled if and only if for every $j\in\beta_1$ the linear program
\begin{equation}
\label{EqLP_j}    \min \mu_j^H\quad\mbox{ subject to }\quad(\mu^h,\mu^g,\mu^H,\mu^G)\in{\cal N}_{VC},\ \mu_i^G\leq 0,\ i\in\beta_1,\ \mu_i^G\geq 0,\ i\in\beta_2
\end{equation}
has a solution and the linear program
\begin{equation}
\label{EqLP_beta2}    \max \sum_{i\in\beta_2}\mu^G_i\quad\mbox{ subject to }\quad(\mu^h,\mu^g,\mu^H,\mu^G)\in{\cal N}_{VC},\ \mu_i^G\leq 0,\ i\in\beta_1,\ \mu_i^G\geq 0,\ i\in\beta_2
\end{equation}
has a solution. Since the feasible regions of these linear programs are not empty, by duality theory of linear programming this is equivalent to the statement that the feasible regions of the corresponding dual programs are not empty. Since the feasible regions of the dual programs to \eqref{EqLP_j} and \eqref{EqLP_beta2}, respectively, are given by \eqref{Eq_z_j} and \eqref{Eq_bar_z}, respectively, the two statements (i) and (ii) are equivalent.
\end{proof}
The characterization of condition \eqref{EqSuffEquMPVC} by Lemma \ref{LemEquAss1} resembles the well-known Mangasarian-Fromovitz constraint qualification of nonlinear programming. It appears to be not very restrictive, e.g. in case when $\beta_1=\emptyset$, $\beta_2=I^{00}$ condition \eqref{EqSuffEquMPVC} is fulfilled when the system
\begin{eqnarray*}
    \nonumber &&\nabla h(\xb)\bar z=0,\\
    \nonumber &&\nabla g_i(\xb)\bar z=0,\ i\in I^g,\\
    &&\nabla G_i(\xb)\bar z=0,\ i\in I^{+0},\\
    \nonumber &&\nabla G_i(\xb)\bar z<0,\ i\in I^{00},\\
    \nonumber &&\nabla H_i(\xb)\bar z=0,\ i\in I^{0-}\cup I^{00}\cup I^{0+}
  \end{eqnarray*}
has a solution. Hence we think that Theorem \ref{ThMPVCRegNormalCone} is likely to be applicable in many situations.

At the end of this section we consider $\Q$-stationarity for MPVC with respect to $\Q=\Q_{VC}$, where
\[\Q_{VC}:=\{(Q^{\beta_1,\beta_2}_{VC},Q^{\beta_2,\beta_1}_{VC})\mv (\beta_1,\beta_2)\mbox{ is partition of }I^{00}\}.\]
\begin{definition}Let $\xb\in\Omega_V$.
\begin{enumerate}
\item   We say that $\xb$ is {\em $\Q$-stationary for the MPVC \eqref{EqMPVC} with respect to the partition $(\beta_1,\beta_2)$ of the index set $I^{00}$} if
  \[0\in \nabla f(\xb)+M^{\beta_1,\beta_2}_{VC},\]
  where $M^{\beta_1,\beta_2}_{VC}$ is given by \eqref{EqInclN_MPVC}.
\item  We say that $\xb$ is {\em $\Q$-stationary for the MPVC \eqref{EqMPVC} } if it is $\Q$-stationary with respect to some  partition $(\beta_1,\beta_2)$ of the index set $I^{00}$.
\item We say that $\xb$ is {\em $\Q_M$-stationary for the MPVC \eqref{EqMPVC}} if there is some partition $(\beta_1,\beta_2)$ of $I^{00}$ such that
\[0\in\nabla f(\xb)+\nabla F(\xb)^T\left((Q^{\beta_1,\beta_2}_{VC})^\circ\cap(\ker\nabla F(\xb)^T+ (Q^{\beta_2,\beta_1}_{VC})^\circ)\cap N_D(F(\xb))\right).\]
\end{enumerate}
\end{definition}
It follows from the definition that
\[\tilde N^{I^{00},\emptyset}_{VC}\subset\{(\lambda^h,\lambda^g,\lambda^H,\lambda^G)\in{\cal R}_{VC}\mv \begin{array}[t]{l} \lambda^g_i\geq 0,\ i\in I^g,\quad
  \lambda^H_i\geq0,\  i\in I^{0-},\\
      \lambda^G_i\geq0,\ i\in I^{+0},\quad\lambda^G_i=0,\ i\in I^{00}\}\subset N_D(F(\xb)).\end{array}\]
Hence, if $\xb$ is $\Q$-stationary with respect to $(I^{00},\emptyset)$, it is automatically $\Q_M$-stationary and  the following theorem follows from Proposition \ref{PropMPVC}, Theorem \ref{ThMPVCRegNormalCone} and Theorem \ref{ThM_betaStat}(1.).
\begin{theorem}\label{ThMPVC_beta_stat}
  Assume that GGCQ is fulfilled at the point $\xb\in\Omega_V$. If $\xb$ is B-stationary, then $\xb$ is $\Q$-stationary for the MPVC \eqref{EqMPVC} with respect to every partition $(\beta_1,\beta_2)$ of $I^{00}$ and, in particular, it is $Q-$stationary with respect to the partition  $(I^{00},\emptyset)$ implying $\Q_M-$stationarity. Conversely, if $\xb$ is $\Q$-stationary with respect to a partition $(\beta_1,\beta_2)$ of $I^{00}$, which fulfills also the assumptions of Theorem \ref{ThMPVCRegNormalCone}, then $\xb$ is S-stationary and consequently B-stationary as well.
\end{theorem}
Further we have
\[\tilde N^{\emptyset, I^{00}}_{VC}\subset {\cal S}:=\{(\lambda^h,\lambda^g,\lambda^H,\lambda^G)\in{\cal R}_{VC}\mv \begin{array}[t]{l} \lambda^g_i\geq 0,\ i\in I^g,\quad
  \lambda^H_i\geq0,\  i\in I^{0-},\\
      \lambda^G_i\geq0,\ i\in I^{+0},\quad\lambda^G_i\geq0,\ \lambda^H_i\geq 0,\ i\in I^{00}\}.\end{array}\]
It was stated in \cite[Theorem 4]{AchKa08} that, under some weak constraint qualification, the condition $0\in\nabla f(\xb)+\nabla F(\xb)^T{\cal S}$ is a necessary condition for a local minimizer. Hence, if $\xb$ is $\Q$-stationary with respect to $(\emptyset,I^{00})$, then it fulfills also the necessary conditions of \cite[Theorem 5.3]{AchKa08}. From Lemma \ref{LemBetaStat} we obtain that $\xb$ is $\Q$-stationary with respect to $(\beta_1,\beta_2)$, if and only if it $\Q$-stationary with respect to $(\beta_2,\beta_1)$. Hence we conclude, that $\Q$-stationarity with respect to $(I^{00},\emptyset)$ implies both $\Q_M$-stationary and the necessary optimality conditions of \cite[Theorem 4]{AchKa08}.

Finally note that GGCQ for MPVC is equivalent to the condition MPVC-GCQ introduced in \cite{Ho09}, where it is also shown in \cite[Theorem 6.1.8]{Ho09} that under MPVC-GCQ any B-stationary point of MPVC is already M-stationary.

\section{Application to generalized equations}
Now we consider the problem
\begin{eqnarray}
\nonumber\min_{(x,y)\in\R^n\times\R^m} &&f(x,y)\\
\nonumber\text{subject to }&&0\in G(x,y)+\widehat N_{\Gamma}(y),\\
\label{EqGE}&&x\in C,
\end{eqnarray}
where the mappings $f:\R^n\times\R^m\to\R$, $G:\R^n\times\R^m\to\R^m$ are assumed to be continuously differentiable, $C$ is a closed subset of $\R^n$ and the set $\Gamma\subset\R^m$ is given by $C^2$ inequalities, i.e. $\Gamma:=\{y\in\R^m\mv g_i(y)\leq 0, i=1,\ldots,l\}$, where $g:\R^m\to\R^l$ is twice continuously differentiable. The constraints fit into our general setting \eqref{EqConstrSyst} with
\begin{equation}
  \label{EqConstrGE}F(x,y):=\left(\begin{array}{c}x\\(y,-G(x,y))\end{array}\right),\ D:=C\times \Gr \widehat N_\Gamma.
\end{equation}
We denote the feasible region of \eqref{EqGE} by $\Omega_{GE}$. We consider a point $(\xb,\yb)\in\Omega_{GE}$, fixed throughout this section, and we suppose  the following assumptions:
\begin{assumption}\label{AssGE}
\begin{enumerate}
  \item The tangent cone $T_C(\xb)$ is convex and $T_D(F(\xb,\yb))=T_C(\xb)\times T_{\Gr\widehat N_\Gamma}(\yb,-G(\xb,\yb))$.
  \item GGCQ holds at $(\xb,\yb)$.
  \item There is some $v\in\R^m$ such that
  \[ \nabla g_i(\yb)v<0,\ i\in\Ib:=\{i\mv g_i(\yb)=0\},\]
  i.e. MFCQ holds at $\yb$.
\end{enumerate}
\end{assumption}
The first assumption is e.g. fulfilled if $C$ is given by $C^1$-inequalities $h_i(x)\leq 0$ $i=1,\ldots,s$ and MFCQ is fulfilled at $\xb$. Note that the third assumption, that MFCQ holds at $\yb$, is only made in order to ease the presentation. We claim that it can be weakened to the weaker assumption of metric regularity in the vicinity of $\yb$ (cf. \cite{GfrOut14a}) or metric subregularity and the bounded extreme point property as used in the recent paper \cite{GfrMo15a}.

In what follows we set $\yba:=-G(\xb,\yb)$ and we define by
\[\Lb:=\{\lambda\in\widehat N_{\R_-^l}(g(\yb))\mv \nabla
g(\yb)^T\lambda=\yba\},\] the set of {\em Lagrange multipliers}
associated with $(\yb,\yba)$ and by
\[\Kb:=T_\Gamma(y)\cap (\yba)^\perp\] the {\em critical cone} to $\Gamma$ at $\yb$ with respect to
$\yba$. Thanks to the assumed MFCQ for the inequalities describing $\Gamma$ we have $T_\Gamma(\yb)=T_\Gamma^{\rm lin}(\yb)=\{v\mv \nabla g_i(\yb)v\leq 0,\ i \in\Ib\}$, $\widehat N_\Gamma(\yb)=\nabla g(\yb)^T\widehat N_{\R^l_-}(g(\yb))$ and that $\Lb\not=\emptyset$ is compact. Note that we do not require that the gradients $\nabla g_i(\yb)$, $i\in\Ib$ are linearly independent and hence the set $\Lb$ can contain more than one element.

Given a multiplier $\lambda\in\widehat N_{\R^l_-}(g(\yb))$ we introduce the
index sets
\[I^+(\lambda):=\{i\in\{1,\ldots,l\}\mv \lambda_i>0\},\
\bar I^0(\lambda):=\Ib\setminus I^+(\lambda).\]
Apart from them we will be
working with
\[\bar I^+:=\bigcup_{\lambda\in \Lb}
I^+(\lambda),\quad \bar I^0:=\Ib\setminus \bar I^+.\]
By convexity of the set $\Lb$ a multiplier $\lambda^+\in\Lb$ verifying $I^+(\lambda^+)=\bar I^+$ exists. Further we have
\begin{equation}
  \label{EqI_Plus}
  \big(\sum_{i\in\Ib}\nabla g_i(\yb)\gamma_i=0,\ \gamma_i\geq 0, i\in \bar I^0\big)\ \Rightarrow\ \gamma_i=0, i\in\bar I^0.
\end{equation}
Indeed, if there would exist numbers $\gamma_i$, $i\in \Ib$ violating \eqref{EqI_Plus}, then, by setting
\[\tilde \lambda_i=\begin{cases}\lambda^+_i+t\gamma_i,&i\in\Ib,\\
 0,&i\not\in\Ib\end{cases}\]
with $t>0$ sufficiently small, we would obtain the contradiction that $\bar I^+$ is strictly contained in $I^+(\tilde\lambda)$.

Note that $\Kb=\{v\mv \nabla g_i(\yb)v=0,i\in \bar I^+,\ \nabla g_i(\yb)v\leq0,i\in \bar I^0\}$, cf. \cite[Lemma 2]{GfrOut14a} and therefore
$\Kb^\circ=\{\sum_{i\in\Ib}\mu_i\nabla g_i(\yb)\mv \mu_i\geq 0,i\in\bar I^0\}$.

 For a direction $v\in\Kb$   we further introduce the
{\em directional multiplier set}
\[\Lb(v):= \mathop{\rm arg\,max}\limits_{\lambda\in
\Lb}v^T\nabla^2(\lambda^Tg)(\yb)v.\]

Application of  \cite[Exercise 13.17, Corollary 13.43(a)]{RoWe98} (see also \cite[Theorem 1]{GfrOut14a}) yields the representation
\begin{equation}\label{EqTanConeNormalConeMap}
T_{\Gr\widehat N_\Gamma}(\yb,\yba)=\{(v,v^\ast)\mv v\in\Kb,\ \exists\lambda\in\Lb(v): v^\ast \in \nabla^2(\lambda^Tg)(\yb)v+\widehat N_{\Kb}(v)\}.
\end{equation}
A description of the regular normal cone  $\widehat N_{\Gr\widehat N_\Gamma}(\yb,\yba)$ can be found in \cite[Theorem 2]{GfrOut14a}.

In general the structure of the tangent cone \eqref{EqTanConeNormalConeMap} is rather complicated. E.g., it is not known whether it always can be represented as the union of finitely many convex polyhedral cones or whether Assumption \ref{AssGE} is sufficient for M-stationarity of a B-stationary point.

In the following theorem we state a sufficient condition that the formula $\widehat N_{\Omega_{GE}}=\nabla F(\xb,\yb)^T\widehat N_D(F(\xb,\yb))$ is valid, i.e., that S-stationarity holds at $(\xb,\yb)$ provided it is B-stationary. We denote by
${\rm lin\,} T_C(\xb)$  the {\em lineality space} of $T_C(\xb)$, i.e. the largest linear space contained in $T_C(\xb)$. Since $T_C(\xb)$ is a closed convex cone by our assumption, we have
${\rm lin\,} T_C(\xb)=T_C(\xb)\cap (-T_C(\xb))$.

\begin{theorem}\label{ThGE_S_Stat1}Assume that Assumption \ref{AssGE} holds and that for every $w\in\Kb$, every $\lambda_w\in\Lb(w)$ and every $z\in\R^m$ verifying
\begin{eqnarray*}&&\nabla g_i(\yb)z=0,\ i\in \bar I^+\\
&&\nabla_x G(\xb,\yb)^Tz\in ({\rm lin\,} T_C(\xb))^\perp
\end{eqnarray*}
one has
\begin{equation}\label{EqOrthogonality}
z^T(\nabla_y G(\xb,\yb)+\nabla^2(\lambda_w^Tg)(\yb))w=0.\end{equation}
Further suppose that there exist some $\tilde u\in\ri T_C(\xb)$, $\tilde w\in\Kb$, $\tilde\lambda\in\Lb(\tilde w)$  and some reals $\tilde\mu_i$, $i\in\Ib$  such that
\begin{equation}\label{EqGE_Slater}\tilde \mu_i>0,\ i\in \bar I^0\text{ and } \nabla_xG(\xb,\yb)\tilde u+\nabla_yG(\xb,\yb)\tilde  w +\nabla^2(\tilde\lambda^Tg)(\yb)\tilde w+\sum_{i\in\Ib}\nabla g_i(\yb)\tilde \mu_i=0.
\end{equation}
Then one has
\begin{equation}\label{EqS_Stat_GE}
\widehat N_{\Omega_{GE}}=\left\{\left(\begin{array}{c}-\nabla_x G(\xb,\yb)^Tw+c^\ast\\-\nabla_y G(\xb,\yb)^Tw+w^\ast\end{array}\right)\mv c^\ast\in\widehat N_C(\xb), (w^\ast,w)\in \widehat N_{\Gr \widehat N_\Gamma}(\yb,\yba)\right\}.
\end{equation}
\end{theorem}
\begin{proof} By Assumption \ref{AssGE} we obtain that
\[T^{\rm lin}_{\Omega_{GE}}(\xb,\yb)=\{(u,v)\mv u\in T_C(\xb), (v,-\nabla_x G(\xb,\yb) u-\nabla_y G(\xb,\yb)v)\in T_{\Gr\widehat N_\Gamma}(\yb,\yba)\}\]
and, together with \eqref{EqTanConeNormalConeMap}, that $Q:=T_C(\xb)\times\{0\}^m\times \Kb^\circ$ is a convex cone contained in $T_D(\xb,\yb,\yba)$.
We shall apply Corollary \ref{CorS_Stat} with this cone $Q$  by showing  that $T_D(\xb,\yb,\yba)=\Tb(Q)$ and that there is some $(u,v)$ such that $\nabla F(\xb,\yb)(u,v)\in\ri\co \Tb(Q)$. In a first step we show  $T_D(\xb,\yb,\yba)=\Tb(Q)$, i.e. we prove that for every $(t,w,w^\ast)\in T_D(\xb,\yb,\yba)$ there is some $q:=(t_q,0,k^\ast)\in Q$  and some $(u,v)\in\R^n\times\R^m$ such that
\begin{equation}\label{EqApplCorrS_Stat}\nabla F(\xb,\yb)\left(\begin{array}{c}u\\v\end{array}\right)=\left(\begin{array}{c}u\\(v,-\nabla_x G(\xb,\yb)u-\nabla_yG(\xb,\yb)v)\end{array}\right)=\left(\begin{array}{c}t-t_q\\(w,w^\ast-k^\ast)\end{array}\right)\in T_D(\xb,\yb,\yba).
\end{equation}
Let $(t,w,w^\ast)\in T_D(\xb,\yb,\yba)$ be arbitrarily fixed and let $w^\ast=\nabla^2(\lambda_w^Tg)(\yb)w+n^\ast$ with $\lambda_w\in\Lb(w)$ and $n^\ast\in \widehat N_{\Kb}(w)$.

Denoting by $A$ the $\vert \bar I^+\vert\times m$ matrix, whose rows are given by $\nabla g_i(\yb)$, $i\in\bar I^+$,  we obtain from \eqref{EqOrthogonality}
that
\begin{eqnarray*}(\nabla_y G(\xb,\yb)+\nabla^2(\lambda_w^Tg)(\yb))w&\in&\left(\ker A\cap (\nabla_x G(\xb,\yb)^{-T}({\rm lin\,} T_C(\xb))^\perp)\right)^\perp\\
&=&\range A^T + \nabla_x G(\xb,\yb)({\rm lin\,} T_C(\xb))\\
&=&-\range A^T - \nabla_x G(\xb,\yb)({\rm lin\,} T_C(\xb)).
\end{eqnarray*}
Hence there is some $\tilde k^\ast\in \range A^T=\Span\{\nabla g_i(\yb)\mv i\in \bar I^+\}$ and some $l\in {\rm lin\,} T_C(\xb)$ such that
$(\nabla_y G(\xb,\yb)+\nabla^2(\lambda_w^Tg)(\yb))w= -\tilde k^\ast -\nabla_x G(\xb,\yb)l$. Setting $t_q:=t-l$, $u:=l$, $v:=w$ and $k^\ast :=n^\ast-\tilde k^\ast$ and taking into account that $n^\ast\in\widehat N_{\Kb}(w)=\Kb^\circ\cap \{w\}^\perp\subset\Kb^\circ$ and that $\Span\{\nabla g_i(\yb)\mv i\in \bar I^+\}$ is exactly the lineality space of $\Kb^\circ$, we have $t_q\in T_C(\xb)$, $k^\ast\in \Kb^\circ$ and
\begin{eqnarray*}-\nabla_x G(\xb,\yb)u-\nabla_yG(\xb,\yb)v=-\nabla_x G(\xb,\yb)l-\nabla_yG(\xb,\yb)w
=\nabla^2(\lambda_w^Tg)(\yb)w+\tilde k^\ast= w^\ast-k^\ast.
\end{eqnarray*}
Thus
\begin{eqnarray*}\lefteqn{\left(\begin{array}{c}u\\(v,-\nabla_x G(\xb,\yb)u-\nabla_yG(\xb,\yb)v)\end{array}\right)=\left(\begin{array}{c}t-t_q\\(w,w^\ast-k^\ast)
\end{array}\right)}\\
&=&\left(\begin{array}{c}l\\(w, \nabla^2(\lambda_w^Tg)(\yb)w +\tilde k^\ast) \end{array}\right)\in T_C(\xb)\times T_{\Gr\widehat N_\Gamma}(\yb,\yba))=T_D(\xb,\yb,\yba)
\end{eqnarray*}
verifying \eqref{EqApplCorrS_Stat}, and therefore $T_D(\xb,\yb,\yba)=\Tb(Q)$ holds.

In order to show that there are $(u,v)$ such that $\nabla F(\xb,\yb)(u,v)\in\ri\co \Tb(Q)$, we observe first that
\begin{equation}\label{EqConvT_D}\co \Tb(Q)=\co T_D(\xb,\yb,\yba)=(\co {\cal S}_1)+{\cal S}_2,\end{equation}
where ${\cal S}_1:=\{(0,w,\nabla^2(\lambda^Tg)(\yb)w)\mv w\in\Kb,\lambda\in\Lb(w)\}$ and  ${\cal S}_2:=T_C(\xb)\times\{0\}^m\times \Kb^\circ$. Indeed, by Assumption \ref{AssGE} and \eqref{EqTanConeNormalConeMap} it can be easily seen that $T_D(\xb,\yb,\yba)\subset {\cal S}_1+{\cal S}_2$
and by convexity of ${\cal S}_2$ the inclusion
\[\co T_D(\xb,\yb,\yba)\subset\co ({\cal S}_1+{\cal S}_2)=(\co {\cal S}_1)+{\cal S}_2\]
readily follows. On the other hand we have ${\cal S}_1,{\cal S}_2\subset T_D(\xb,\yb,\yba)$ implying $\co{\cal S}_1,{\cal S}_2\subset \co T_D(\xb,\yb,\yba)$ and, together with the fact that $\co T_D(\xb,\yb,\yba)$ is a convex cone, the reverse inclusion
\[\co T_D(\xb,\yb,\yba)\supset(\co {\cal S}_1)+{\cal S}_2\]
follows as well and the validity of \eqref{EqConvT_D} is shown.

Now consider $(0,w,w^\ast)\in\ri\co{\cal S}_1$. Then there are nonnegative coefficients $\alpha_j\geq 0$, $j=1,\ldots,s$, $\sum_{j=1}^s\alpha_j=1$ and elements $(0,w_j,w_j^\ast)\in {\cal S}_1$ such that $(0,w,w^\ast)=\sum_{j=1}^s\alpha_j(0,w_j,w_j^\ast)$. Then, by proceeding as before, for every $j=1,\ldots,s$ we can find $\tilde k_j^\ast\in \Span\{\nabla g_i(\yb)\mv i\in \bar I^+\}$ and $l_j\in {\rm lin\,} T_C(\xb)$ such that
\[-\nabla_x G(\xb,\yb)l_j-\nabla_yG(\xb,\yb)w_j =w_j^\ast+\tilde k_j^\ast\]
By setting $l:=\sum_{j=1}^s\alpha_jl_j$, $u:=l+\tilde u$, $v:=w+\tilde w$, $\tilde w^\ast:=\nabla^2(\tilde\lambda^Tg)(\yb)\tilde w$, $k^\ast:=\sum_{j=1}^s\alpha_j\tilde k_j^\ast+\sum_{i\in\Ib}\nabla g_i(\yb)\tilde \mu_i$,  we obtain
\begin{eqnarray*}\lefteqn{\nabla F(\xb,\yb)\left(\begin{array}{c}u\\v\end{array}\right)}\\
&=&\left(\begin{array}{c}u\\(v,-\nabla_x G(\xb,\yb)u-\nabla_y G(\xb,\yb)v)\end{array}\right)
=\left(\begin{array}{c}0\\(w+\tilde w, w^\ast+\tilde w^\ast) \end{array}\right)+ \left(\begin{array}{c}\tilde u+ l\\(0,k^\ast) \end{array}\right).
\end{eqnarray*}
Since $\sum_{i\in\Ib}\nabla g_i(\yb)\tilde \mu_i\in\ri \Kb^\circ$ by \cite[Theorem 6.6]{Ro70}, $\sum_{j=1}^s \alpha_j \tilde k_j^\ast\in\Span\{\nabla g_i(\xb)\mv i\in\bar I^+\}\subset\lin \Kb^\circ$, $\tilde u\in\ri T_C(\xb)$ and $l\in\lin T_C(\xb)$, we conclude
\begin{eqnarray*}\left(\begin{array}{c}\tilde u+ l\\(0, k^\ast) \end{array}\right)\in \ri {\cal S}_2 +\lin {\cal S}_2=\ri {\cal S}_2.\end{eqnarray*}
Further, since $(0,w,w^\ast)\in\ri \co{\cal S}_1$, $(0,\tilde w,\tilde w^\ast)\in{\cal S}_1$ and ${\cal S}_1$ is a cone, we obtain $(0,w+\tilde w,w^\ast+\tilde w^\ast)\in \ri\co{\cal S}_1$.
Thus, by taking into account \cite[Corollary 6.6.2]{Ro70},
\begin{eqnarray*}
\nabla F(\xb,\yb)\left(\begin{array}{c}u\\v\end{array}\right)&\in&\ri \co {\cal S}_1+\ri {\cal S}_2 =\ri((\co {\cal S}_1)+ {\cal S}_2)=\ri \co T_D(\xb,\yb,\yba)\end{eqnarray*}
and this finishes the proof.
\end{proof}
\begin{remark}
Theorem \ref{ThGE_S_Stat1} improves \cite[Theorem 5]{GfrOut14a}, where the assumption
\[\nabla_x G(\xb,\yb)(\lin T_C(\xb))+\Span\{\nabla g_i(\xb)\mv i\in\bar I^+\}=\R^m\]
is used. Note that this assumption is equivalent to $\{0\}^m=(\Span\{\nabla g_i(\xb)\mv i\in\bar I^+\})^\perp\cap \big(\nabla_x G(\xb,\yb)(\lin T_C(\xb))\big)^\perp$ and thus the only element $z$ with $\nabla g_i(\xb)z=0$, $i\in\bar I^+$ and $\nabla_x G(\xb,\yb)^Tz\in (\lin T_C(\xb))^\perp$ is $z=0$ and therefore \eqref{EqOrthogonality} trivially holds. Further, this assumption also implies \eqref{EqGE_Slater}, because for arbitrary $u\in\ri T_C(\xb)$ and $\tilde\mu_i>0$, $i\in\bar I^0$, we can find $l\in \lin T_C(\xb)$ and $\tilde\mu_i$, $i\in\bar I^+$ with $\nabla_x G(\xb,\yb)l+\sum_{i\in\bar I^+}\tilde\mu_i\nabla g_i(\xb)=-\nabla_x G(\xb,\yb) u-\sum_{i\in\bar I^0}\tilde\mu_i\nabla g_i(\xb)$ and now \eqref{EqGE_Slater} follows with $\tilde u= u+l\in \ri T_C(\xb)$, $\tilde w=0$.
\end{remark}
Next we consider $\Q$-stationarity  for the problem \eqref{EqGE} under  an additional assumption which allows a simplified description of the contingent cone  $T_{\Gr\widehat N_\Gamma}(\yb,\yba)$ as stated in \cite[Theorem 3]{GfrOut14a}.
%
\begin{theorem}\label{ThNormConeMapConstLambd} Assume that Assumption \ref{AssGE}(3.) holds at $\yb$.
Further assume that $\Lb(v_1)=\Lb(v_2)$ $\forall 0\not=v_1,v_2\in \Kb$ and let $\bar\lambda$ be an arbitrary
multiplier from $\Lb(v)$ for some $0\not= v\in\Kb$, if
$\Kb\not=\{0\}$ and $\bar\lambda\in\Lb$ otherwise. Then
 \begin{equation}\label{EqTangConeConstLambda}
 T_{\Gr \widehat N_\Gamma}(\yb,\yba)=\{(v,v^\ast)\mv v\in \Kb,\ v^\ast\in \nabla^2(\bar\lambda^Tg)(\yb)v+\widehat N_{\Kb}(v)\}
 \end{equation}
 and
 \begin{equation}\label{EqNormalConeConstLambda}
 \widehat N_{\Gr \widehat N_\Gamma}(\yb,\yba)=\{(w^\ast,w)\mv
  w\in\Kb, w^\ast\in -\nabla^2(\bar\lambda^Tg)(\yb)w+ \Kb^\circ\}.
 \end{equation}
\end{theorem}
The assumption $\Lb(v_1)=\Lb(v_2)$ $\forall 0\not=v_1,v_2\in \Kb$ is for instance fulfilled, if the inequalities $g_i(y)\leq 0$ fulfill the {\em constant rank constraint qualification} at $\yb$, see e.g. \cite[Corollary 3.2]{HenKruOut13}.

In what follows we will assume that the assumptions of Theorem \ref{ThNormConeMapConstLambd} hold and that the tangent cone $T_C(\xb)$ is a convex polyhedral cone. For every index set $\beta\subset \bar I^0$ we define the convex polyhedral cone
\begin{equation}\label{EqQ_MPEC}
  Q^\beta_{GE}:=T_C(\xb)\times \{(v,v^\ast)\mv (v, v^\ast- \nabla^2(\bar \lambda^Tg)(\yb)v)\in K_\beta \times K_{\beta}^\ast\},
\end{equation}
where
\[K_\beta:=\left\{v\mv \nabla g_i(\yb)v\begin{cases}=0,& i\in \bar I^+\cup \beta,\\
\leq 0,& i\in\bar I^0\setminus\beta\end{cases}\right\},\quad K_{\beta}^\ast:=\left\{\sum_{i\in \bar I^+\cup\beta}\mu_i\nabla g_i(\yb)\mv \mu_i\geq 0,\ i\in \beta\right\}.\]
Then we have
\[(K_\beta\times K_\beta^\ast)^\circ=\left\{\sum_{i\in\Ib}\mu_i\nabla g_i(\yb)\mv \mu_i\geq 0, i\in\bar I^0\setminus\beta\right\}\times\left\{z\mv \nabla g_i(\yb)z\begin{cases}=0,& i\in \bar I^+,\\
\leq 0,& i\in\beta\end{cases}\right\}\]
and
\begin{eqnarray*}(Q^\beta_{GE})^\circ&=&\widehat N_C(\xb)\times \left(\begin{array}{cc}I&-\nabla^2(\lb^Tg)(\yb)\\0&I\end{array}\right)(K_\beta \times K_{\beta}^\ast)^\circ\\
&=&\widehat N_C(\xb)\times \{(w^\ast,w)\mv (w^\ast+\nabla^2(\lb^Tg)(\yb)w,w)\in(K_\beta \times K_{\beta}^\ast)^\circ\}.\end{eqnarray*}
It is easy to see that under the assumptions of Theorem \ref{ThNormConeMapConstLambd} we have
\begin{equation}\label{EqTan_GE_Union_Q}T_D(F(\xb,\yb))=\bigcup_{\beta\subset \bar I^0}Q^\beta_{GE}\end{equation}
and thus
\[\widehat N_D(F(\xb,\yb))=\widehat N_C(\xb)\times \widehat N_{\Gr \widehat N_\Gamma}(\yb,\yba)=\bigcap_{\beta\subset \bar I^0}(Q^\beta_{GE})^\circ.\]
Note that for every pair $(\beta_1,\beta_2)\subset \bar I^0\times\bar I^0$ the cones $(Q^{\beta_1}_{GE},Q^{\beta_2}_{GE})$ fulfill \eqref{EqCondQi} because they are convex polyhedral cones.

\begin{proposition}\label{PropGE}
  Let $(\xb,\yb)\in\Omega_{GE}$ and assume in addition to Assumption \ref{AssGE} that the contingent cone $T_C(\xb)$ is polyhedral and $\Lb(v_1)=\Lb(v_2)\ \forall 0\not=v_1,v_2\in\Kb$. Then for every pair $(\beta_1,\beta_2)\subset \bar I^0\times\bar I^0$ we have
\begin{eqnarray}\label{EqInclN_MPGE}
  \widehat N_{\Omega_{GE}}(\xb,\yb)\subset \nabla F(\xb,\yb)^T\tilde N^{\beta_1,\beta_2}_{GE}=\left\{\left(\begin{array}{c}\eta_C -\nabla_xG(\xb,\yb)^Tq\\q^\ast-\nabla_yG(\xb,\yb)^Tq\end{array}\right)\mv
  (\eta_C,q^\ast,q)\in\tilde N^{\beta_1,\beta_2}_{GE}\right\}=:   M^{\beta_1,\beta_2}_{GE}
\end{eqnarray}
where
\begin{subequations}
\begin{align}
\nonumber  \tilde N^{\beta_1,\beta_2}_{GE}
:=\Big\{&(\eta_C,q^\ast,q)\in \widehat N_C(\xb)\times\R^m\times\R^m\mv
\exists r\in\R^m,\mu_i^r\mu_i^q, i\in\Ib:\\
    &q^\ast+\nabla^2(\bar\lambda^Tg)(\yb)q=\sum_{i\in\Ib}\mu_i^q\nabla g_i(\yb),\label{SubEq_q_star}\\
    &\nabla g_i(\yb)q=0,\ i\in\bar I^+,\ \nabla g_i(\yb)q\leq 0,\ i\in\beta_1,\ \mu_i^q\geq 0,\ i\in\bar I^0\setminus \beta_1,\label{SubEqBeta1}\\
     &\nabla g_i(\yb)r=0,\ i\in\bar I^+,\ \nabla g_i(\yb)q\leq \nabla g_i(\yb)r,\ i\in\beta_2,\ \mu_i^q\geq \mu_i^r,\ i\in\bar I^0\setminus \beta_2, \label{SubEqBeta2}\\
     &\nabla_y G(\xb,\yb)^Tr+\nabla^2(\bar\lambda^Tg)(\yb)r=\sum_{i\in\Ib}\mu_i^r\nabla g_i(\yb),\label{SubEq_r}\\
     &\eta_C\in \nabla_x G(\xb,\yb)^Tr+\widehat N_C(\xb)\Big\}\label{SubEqEtaC}
\end{align}
\end{subequations}
and $\bar\lambda$ is an arbitrarily fixed multiplier from $\Lb(v)$ for some $0\not=v\in\Kb$, if $\Kb\not=\{0\}$ and $\bar\lambda\in\Lb$ otherwise.
\end{proposition}
\begin{proof}
The statement follows immediately from Theorem \ref{ThQ} if we can show
\begin{equation}\tilde N^{\beta_1,\beta_2}_{GE}=(Q^{\beta_1}_{GE})^\circ\cap(\ker\nabla F(\xb,\yb)^T+(Q^{\beta_2}_{GE})^\circ)
\label{EqN_beta_GE}.
\end{equation}
Consider an element $(\eta_C,q^\ast,q)\in (Q^{\beta_1}_{GE})^\circ\cap(\ker\nabla F(\xb,\yb)^T+(Q^{\beta_2}_{GE})^\circ)$. Then there are elements $(\rho_C,r^\ast,r)\in \ker \nabla F(\xb,\yb)^T$ and $(\tilde\eta_C,\tilde q^\ast,\tilde q)\in (Q^{\beta_2}_{GE})^\circ$ such that
\begin{equation*}
  \left(\begin{array}{c}\eta_C\\(q^\ast,q)\end{array}\right)=\left(\begin{array}{c}\rho_C\\(r^\ast,r)\end{array}\right)+
  \left(\begin{array}{c}\tilde\eta_C\\(\tilde q^\ast,\tilde q)\end{array}\right).
\end{equation*}
Since
\[\nabla F(\xb,\yb)^T\left(\begin{array}{c}\rho_C\\r^\ast\\r\end{array}\right)=\left(\begin{array}{c}\rho_C -\nabla_xG(\xb,\yb)^Tr\\r^\ast-\nabla_yG(\xb,\yb)^Tr\end{array}\right)=\left(\begin{array}{c}0\\0\end{array}\right),\]
we obtain $\rho_C=\nabla_xG(\xb,\yb)^Tr=\eta_C-\tilde\eta_C$ and thus $\eta_C=\nabla_xG(\xb,\yb)^Tr+\tilde\eta_C\in \nabla_xG(\xb,\yb)^Tr+\widehat N_C(\xb)$ verifying \eqref{SubEqEtaC}. The relations \eqref{SubEq_q_star} and \eqref{SubEqBeta1} follow simply from the representation of $(Q^{\beta_1}_{GE})^\circ$. By using the representations  $\tilde q^\ast+\nabla^2(\bar\lambda^Tg)(\yb)\tilde q=\sum_{i\in\Ib}\mu_i^{\tilde q}\nabla g_i(\yb)$ with $\mu_i^{\tilde q}\geq 0$, $i\in\bar I^0\setminus\beta_2$, it follows that
\[ 0=r^\ast+ \tilde q^\ast-q^\ast  =\nabla_y G(\xb,\yb)^Tr+\nabla^2(\bar\lambda^Tg)(\yb)(q-\tilde q)-\sum_{i\in\Ib}(\mu_i^q-\mu_i^{\tilde q})\nabla g_i(\yb).\]
Since $r=q-\tilde q$,  $\nabla g_i(\yb)\tilde q=\nabla g_i(\yb)q=0, i\in\bar I^+$  we have
\begin{eqnarray}
\label{EqCond_r1}   &0=\nabla_y G(\xb,\yb)^Tr+\nabla^2(\bar\lambda^Tg)(\yb)r-\sum_{i\in\Ib}\mu_i^r\nabla g_i(\yb),&\\
\label{EqCond_r2}  &\nabla g_i(\yb)r=0, i\in\bar I^+,&
\end{eqnarray}
where $\mu_i^r:=\mu_i^q-\mu_i^{\tilde q}$, showing \eqref{SubEq_r}.
By taking into account $\nabla g_i(\yb)(q-r)=\nabla g_i(\yb)\tilde q\leq 0, i\in\beta_2$, $\mu_i^q-\mu_i^r=\mu_i^{\tilde q}\geq 0, i\in\bar I^0\setminus\beta_2$ we obtain together with \eqref{EqCond_r2} that \eqref{SubEqBeta2} also holds.
Hence, $(\eta_C,q^\ast,q)$ belongs to the set $\tilde N^{\beta_1,\beta_2}_{GE}$ and the inclusion $(Q^{\beta_1}_{GE})^\circ\cap(\ker\nabla F(\xb,\yb)^T+(Q^{\beta_2}_{GE})^\circ)\subset \tilde N^{\beta_1,\beta_2}_{GE}$ follows.

To show the reverse inclusion consider $(\eta_C,q^\ast,q)\in\tilde N^{\beta_1,\beta_2}_{GE}$ together with $r\in\R^m,\mu_i^q,\mu_i^r,i\in\Ib$ according to the definition. By setting $\rho_C:=\nabla_x G(\xb,\yb)^Tr$, $r^\ast:=\nabla_y G(\xb,\yb)^Tr$,
$(\tilde \eta_C,\tilde q^\ast,\tilde q):=(\eta_C,q^\ast,q)-(\rho_C,r^\ast,r)$ it follows, by using the same arguments as above, that $(\rho_C,r^\ast,r)\in\ker\nabla F(\xb,\yb)^T$ and $(\tilde \eta_C,\tilde q^\ast,\tilde q)\in (Q^{\beta_2}_{GE})^\circ$. Since we obviously have $(\eta_C,q^\ast, q)\in (Q^{\beta_1}_{GE})^\circ$,
we obtain $(\eta_C,q^\ast, q)\in(Q^{\beta_1}_{GE})^\circ\cap(\ker\nabla F(\xb,\yb)^T+(Q^{\beta_2}_{GE})^\circ)$ and this finishes the proof.
\end{proof}
\begin{theorem}\label{THGERegNormCone}Assume that the assumptions of Proposition \ref{PropGE} are fulfilled and assume that we are given a partition $(\beta_1,\beta_2)$ of $\bar I^0$ such that the following two conditions are fulfilled:

{\bf (i)} For every $j\in\beta_2$ there are $l^j\in\lin T_C(\xb)$, $\tilde \alpha^j_i$, $i\in \bar I^+$ and $z^j\in\R^m$ with
\begin{eqnarray*}
\label{EqSuffEquGE1}    &&\nabla g_j(\yb)-\sum_{i\in \bar I^+}\nabla g_i(\yb)\tilde\alpha^j_i  - \big(\nabla_y G(\xb,\yb)+\nabla^2(\bar\lambda^Tg)(\yb)\big)z^j
    +\nabla_x G(\xb,\yb)l^j=0,\\
\label{EqSuffEquGE2}    &&\nabla g_i(\yb)z^j=0, i\in\Ib.
\end{eqnarray*}

{\bf (ii)} For every $k\in\beta_1$ there are $l^k\in\lin T_C(\xb)$, $\tilde \alpha^k_i$, $i\in \bar I^+$ and $z^k\in\R^m$ with
\begin{eqnarray*}
    &&\sum_{i\in \bar I^+}\nabla g_i(\yb)\tilde\alpha^k_i+ \big(\nabla_y G(\xb,\yb)+\nabla^2(\bar\lambda^Tg)(\yb)\big)z^k-\nabla_x G(\xb,\yb)l^k=0,\\
    &&\nabla g_i(\yb)z^k=0, i\in\Ib \setminus\{k\},\  \nabla g_k(\yb)z^k=-1.
\end{eqnarray*}

Then
\[\widehat N_{\Omega_{GE}}(\xb,\yb)=M^{\beta_1,\beta_2}_{GE}=\nabla F(\xb,\yb)^T\widehat N_D(F(\xb,\yb)).\]
\end{theorem}
\begin{proof}
  In view of Theorem \ref{ThQ} and Proposition \ref{PropGE} the statement follows if we can  show $\tilde N^{\beta_1,\beta_2}_{GE}\subset\widehat N_D(F(\xb,\yb))$.
  This inclusion holds true if for every $(\eta_C,q, r)\in\R^n\times\R^m\times\R^m$, $\mu_i^q,\mu_i^r, i\in\Ib$ fulfilling the system
  \begin{equation}\label{EqLinSystGE}
  \begin{array}{l}
    \nabla g_i(\yb)(q-r)=0,\ \nabla g_i(\yb)r=0,\ i\in\bar I^+,\\
    \nabla g_i(\yb)q\leq 0,\ i\in\beta^1,\mu_i^q-\mu_i^r\geq 0,\ i\in\bar I^0\setminus \beta_2=\beta_1,\\
     \nabla g_i(\yb)(q-r)\leq 0,\ i\in\beta^2,\ \mu_i^q\geq 0,\ i\in\bar I^0\setminus \beta_1=\beta_2,\\
     (\nabla_y G(\xb,\yb)^T+\nabla^2(\bar\lambda^Tg)(\yb))r-\sum_{i\in\Ib}\mu_i^r\nabla g_i(\yb)=0,\\
     \eta_C\in\widehat N_C(\xb),\ \eta_C-\nabla_x G(\xb,\yb)^Tr\in\widehat N_C(\xb)
  \end{array}
  \end{equation}
  we have $\nabla g_i(\yb)r\leq 0$, $i\in\beta_2$ and $\mu_i^r\geq0$, $i\in\beta_1$
because  then we have $\nabla g_i(\yb)q\leq 0$, $\mu_i^q\geq 0$, $i\in\beta_1\cup\beta_2=\bar I^0$ and thus the triple $(\eta_C,q^\ast,q)\in \tilde N^{\beta_1,\beta_2}_{GE}$ with $q^\ast=-\nabla^2(\lb^Tg)(\yb)q+\sum_{i\in\Ib}\nabla g_i(\yb)\mu_i^q=-\nabla^2(\lb^Tg)(\yb)q+\sum_{i\in\Ib}\nabla g_i(\yb)\tilde\mu_i$ also belongs to $\widehat N_D( F(\xb,\yb))$.

The first condition $\nabla g_i(\yb)r\leq 0$, $i\in\beta_2$ is equivalent to the requirement that for every $j\in\beta_2$ the optimization problem
\begin{equation}\label{EqAuxLP_GE}\max_{\eta_C,q,r,\mu^q,\mu^r}\nabla q_j(\yb)r\quad\mbox{ subject to \eqref{EqLinSystGE}}\end{equation}
has a solution. Since the tangent cone $T_C(\xb)$ is assumed to be convex polyhedral, so also is the regular normal cone and therefore this program can be written as a linear program for which obviously the trivial solution is  feasible. Hence, by the duality theory of linear programming the program \eqref{EqAuxLP_GE} has a solution, if and only if its dual program has a feasible solution, i.e. there are multipliers $\alpha^j_i, \tilde \alpha^j_i$, $i\in \bar I^+$, $\gamma^j_i\geq 0$, $\tilde\gamma^j_i\geq 0$, $i\in\beta_1$,  $\delta^j_i\geq 0$, $\tilde\delta^j_i\geq 0$, $i\in\beta_2$,  $z^j\in\R^m$ and $\tilde l^j,l^j\in (\widehat N_C(\xb))^\circ=T_C(\xb)$ such that
\begin{eqnarray*}
    &&\tilde l^j+l^j=0,\\
    &&\sum_{i\in\bar I^+}\nabla g_i(\yb)\alpha^j_i+\sum_{i\in\beta_1}\nabla g_i(\yb)\gamma^j_i+\sum_{i\in\beta_2}\nabla g_i(\yb)\delta^j_i=0,\\
    &&-\nabla g_j(\yb)-\sum_{i\in \bar I^+}\nabla g_i(\yb)(\alpha^j_i-\tilde\alpha^j_i)-\sum_{i\in\beta_2}\nabla g_i(\yb)\delta^j_i\\
    && \qquad\qquad+ \big(\nabla_y G(\xb,\yb)+\nabla^2(\bar\lambda^Tg)(\yb)\big)z^j-\nabla_x G(\xb,\yb)l^j=0,\\
    &&\tilde \gamma^j_i=0, i\in \beta_1,\quad \tilde\delta^j_i=0, i\in \beta_2,\\
    &&-\nabla g_i(\yb)z^j=0, i\in\bar I^+\cup\beta_2,\ -\nabla g_i(\yb)z^j+\tilde\gamma^j_i=0, i\in\beta_1.
\end{eqnarray*}
Hence $l^j=-\tilde l^j\in T_C(\xb)\cap (-T_C(\xb))=\lin T_C(\xb)$ and by \eqref{EqI_Plus} we obtain $\gamma^j_i=0$, $i\in\beta_1$ and  $\delta^j_i=0$, $i\in\beta_2$. Now it is easy to see that the dual program to \eqref{EqAuxLP_GE} is feasible if and only if condition (i) is fulfilled.

The second requirement $\mu_i^r\geq 0$, $i\in\beta_1$ is equivalent to the condition that for every $k\in\beta_1$ the program
\begin{equation}\label{EqAuxLP2_GE}\min_{\eta_C,q,r,\mu^q,\mu^r}\mu_k^r\quad\mbox{ subject to \eqref{EqLinSystGE}}\end{equation}
has a solution. Using similar arguments as above we obtain that this is equivalent with the existence of multipliers $\tilde \alpha^k_i$, $i\in \bar I^+$, $\tilde\gamma^k_i\geq 0$, $i\in \beta_1$, $\tilde\delta^k_i\geq 0$, $i\in \beta_2$, $z^k\in\R^m$ and $l^k\in\lin T_C(\xb)$ verifying
\begin{eqnarray*}
    &&\sum_{i\in \bar I^+}\nabla g_i(\yb)\tilde\alpha^k_i+ \big(\nabla_y G(\xb,\yb)+\nabla^2(\bar\lambda^Tg)(\yb)\big)z^k-\nabla_x G(\xb,\yb)l^k=0,\\
    &&-\tilde \gamma^k_i=0, i\in \beta_1,\quad -\tilde\delta^k_i=0, i\in \beta_2,\\
    &&-\nabla g_i(\yb)z^k=0, i\in\bar I^{+}\cup\beta_2,\ -\nabla g_i(\yb)z^k+\tilde\gamma^k_i=0, i\in\beta_1\setminus\{k\},\  -1-\nabla g_k(\yb)z^k+\tilde\gamma^k_k=0
\end{eqnarray*}
and it is easy to see that this is equivalent to condition (ii).
\end{proof}
In order to introduce a suitable $\Q$-stationarity concept for generalized equations, let us define
\[
{\cal B}_{GE}:=\left\{\beta\subset \bar I^0\mv \exists z\in\R^m:\ \nabla g_i(\yb)z
\begin{cases}
  =0,&i\in\bar I^+\cup\beta\\<0,&i\in\bar I^0\setminus\beta
\end{cases}
\right\}
\]
and
\[\Q_{GE}:=\{(Q^{\beta_1}_{GE},\Q^{\beta_2}_{GE})\mv (\beta_1,\beta_2)\in{\cal B}_{GE}\times{\cal B}_{GE}, \beta_1\cup\beta_2=\bar I^0\}.\]
Note that if a subset $\beta\subset\bar I^0$ does not belong to ${\cal B}_{GE}$, then the set $\bar\beta:=\{i\in\bar I^0\mv \nabla g_i(\yb)z=0 \forall z\in K_\beta\}$ fulfills $\beta\subset\bar\beta\in {\cal B}_{GE}$ and $K_\beta=K_{\bar\beta}$. It follows that $K_\beta^\ast\subset K_{\bar\beta}^\ast$ and consequently $Q^{\beta}_{GE}\subset Q^{\bar\beta}_{GE}$. Since we want to consider closed convex cones $Q$ which are as large as possible, we can discard $Q^{\beta}_{GE}$ from our analysis.

It follows immediately from the definition that $\bar I^0\in {\cal B}_{GE}$. Further, by \cite[Lemma 2]{GfrOut14a} we have $\emptyset\in {\cal B}_{GE}$.

In contrast to MPCC and MPVC the condition $Q_1^\circ\cap Q_2^\circ=\widehat N_D(F(\xb,\yb))$ does not hold automatically for every pair $(Q_1,Q_2)\in\Q_{GE}$, but it holds for instance for the pair $(Q^{\bar I^0}_{GE},Q^{\emptyset}_{GE})$.
\begin{definition}Let $(\xb,\yb)\in\Omega_{GE}$.
\begin{enumerate}
\item   We say that $(\xb,\yb)$ is {\em $\Q$-stationary for the program \eqref{EqGE} with respect to the pair $(\beta_1,\beta_2)\in {\cal B}_{GE}\times{\cal B}_{GE}$ satisfying $\beta_1\cup\beta_2=\bar I^0$} if
  \[0\in \nabla f(\xb,\yb)+M^{\beta_1,\beta_2}_{GE},\]
  where $M^{\beta_1,\beta_2}_{GE}$ is given by \eqref{EqInclN_MPGE}.
\item  We say that $(\xb,\yb)$ is {\em $\Q$-stationary for the program \eqref{EqGE}} if it is $\Q$-stationary with respect to some  pair $(\beta_1,\beta_2)\in{\cal B}_{GE}\times{\cal B}_{GE}$ with $\beta_1\cup\beta_2=\bar I^0$.
\item We say that $(\xb,\yb)$ is {\em $\Q_M$-stationary for the program \eqref{EqGE}} if there is some pair $(\beta_1,\beta_2)\in {\cal B}_{GE}\times{\cal B}_{GE}$ with $\beta_1\cup\beta_2=\bar I^0$ such that
\[0\in\nabla f(\xb,\yb)+\nabla F(\xb,\yb)^T\left((Q^{\beta_1}_{GE})^\circ\cap\big(\ker\nabla F(\xb,\yb)^T+ (Q^{\beta_2}_{GE})^\circ\big)\cap N_D(F(\xb,\yb))\right).\]
\end{enumerate}
\end{definition}

By using Proposition \ref{PropGE}, Theorem \ref{THGERegNormCone} and Theorem \ref{ThM_betaStat}(3.) we obtain the following Theorem.
\begin{theorem}
Assume that the assumptions of Proposition \ref{PropGE} hold at the B-stationary point $(\xb,\yb)\in\Omega_{GE}$. Then $(\xb,\yb)$ is $\Q$-stationary with respect to every pair $(\beta_1,\beta_2)\in {\cal B}_{GE}\times{\cal B}_{GE}$ with $\beta_1\cup\beta_2=\bar I^0$ and $(\xb,\yb)$ is also $\Q_M$-stationary. Conversely, if $(\xb,\yb)$ is $\Q$-stationary with respect to some pair $(\beta_1,\beta_2)\in {\cal B}_{GE}\times{\cal B}_{GE}$ fulfilling the assumptions of Theorem \ref{THGERegNormCone}, then $(\xb,\yb)$ is S-stationary and consequently B-stationary as well.
\end{theorem}
\section*{Acknowledgements}
The research  was supported by the Austrian Science Fund (FWF) under grant P26132-N25.
The authors would like to express their gratitude to the reviewers for their careful reading and numerous important suggestions.

\end{document}